\documentclass[12pt]{amsart} 
\usepackage{amsmath,amssymb,amscd,amsthm}
\usepackage{xcolor}
\usepackage{latexsym}
\usepackage{graphicx}
\usepackage[english]{babel}
\usepackage[latin1]{inputenc}       
\usepackage{pgf,pgfarrows,pgfnodes,pgfautomata,pgfheaps}
\usepackage{colortbl}

\usepackage[a4paper,twoside,left=3cm,right=2.8cm,top=3.1cm,bottom=2.3cm]{geometry}

\newtheorem{thm}{Theorem}

\newtheorem{cor}[thm]{Corollary}

\newtheorem{theorem}{Theorem}[section]
\newtheorem{proposition}[theorem]{Proposition}
\newtheorem{lemma}[theorem]{Lemma}

\newtheorem{definition}[theorem]{Definition}

\newtheorem{remark}[theorem]{Remark}

\newtheorem{conjecture}[theorem]{Conjecture}

\def\irr#1{{\rm Irr}(#1)}
\def\irrr#1#2 {\irr {#1 \mid #2}}
\providecommand{\norm}[1]{\left\lVert#1\right\rVert}\newcommand{\R}{\mathbb R}

\newcommand{\sfe}{{{\mathbb S}^{n-1}}}

\begin{document}

\title[On the $L_p$-Brunn-Minkowski and dimensional Brunn-Minkowski]{On the $L_p$-Brunn-Minkowski and dimensional Brunn-Minkowski conjectures for log-concave measures}
\author[Johannes Hosle, Alexander V. Kolesnikov, Galyna V. Livshyts]{Johannes Hosle, Alexander V. Kolesnikov, Galyna V. Livshyts}
\address{Department of Mathematics, University of California, Los Angeles, CA}
\email{jhosle@ucla.edu}
\address{National Research University Higher School of Economics, Russian Federation}
\email{sascha77@mail.ru}
\address{School of Mathematics, Georgia Institute of Technology, Atlanta, GA} \email{glivshyts6@math.gatech.edu}

\subjclass[2010]{Primary: 52} 
\date{\today}
\begin{abstract} We study several of the recent conjectures in regards to the role of symmetry in the inequalities of Brunn-Minkowski type, such as the $L_p$-Brunn-Minkowski conjecture of B\"or\"oczky, Lutwak, Yang and Zhang, and the Dimensional Brunn-Minkowski conjecture of Gardner and Zvavitch, in a unified framework. We obtain several new results for these conjectures. 

We show that when $K\subset L,$ the multiplicative form of the $L_p$-Brunn-Minkowski conjecture holds for Lebesgue measure for $p\geq 1-Cn^{-0.75}$, which improves upon the estimate of Kolesnikov and Milman in the partial case when one body is contained in the other.

We also show that the multiplicative version of the $L_p$-Brunn-Minkowski conjecture for the standard Gaussian measure holds in the case of sets containing sufficiently large ball (whose radius depends on $p$). In particular, the Gaussian Log-Brunn-Minkowski conjecture holds when $K$ and $L$ contain $\sqrt{0.5 (n+1)}B_2^n.$

We formulate an a-priori stronger conjecture for log-concave measures, extending both the $L_p$-Brunn-Minkowski conjecture and the Dimensional one, and verify it in the case when the sets are dilates and the measure is Gaussian. We also show that the Log-Brunn-Minkowski conjecture, if verified, would yield this more general family of inequalities. 

Our results build up on the methods developed by Kolesnikov and Milman as well as Colesanti, Livshyts, Marsiglietti. We furthermore verify that the local version of these conjectures implies the global version in the setting of general measures, and this step uses methods developed recently by Putterman. 
\end{abstract}
\maketitle

\section{Introduction}

Recall that a measure $\mu$ on $\R^n$ is called log-concave if for all non-empty Borel sets $K, L$, and for any $\lambda\in [0,1],$
\begin{equation}\label{log-concave}
\mu(\lambda K+(1-\lambda)L)\geq \mu(K)^{\lambda}\mu(L)^{1-\lambda}
\end{equation}

In accordance with Borell's result \cite{bor}, if a measure $\mu$ has density $e^{-V(x)}$, where $V(x)$ is a convex function on $\R^n$ with non-empty support, then $\mu$ is log-concave. Examples of log-concave measures include Lebesgue volume $|\cdot |$ and the Gaussian measure $\gamma.$

A notable partial case of Borell's theorem is the Brunn-Minkowski inequality, proved in the full generality by Lusternik \cite{Lust}:
\begin{equation}\label{BM}
|\lambda K+(1-\lambda)L|\geq |K|^{\lambda}|L|^{1-\lambda},
\end{equation}
which holds for all Borel-measurable sets $K, L$ and any $\lambda\in [0,1].$ Furthermore, due to the $n-$homogeneity of Lebesgue measure, (\ref{BM}) self-improves to an a-priori stronger form
\begin{equation}\label{BM-add}
|\lambda K+(1-\lambda)L|^{\frac{1}{n}}\geq \lambda |K|^{\frac{1}{n}}+(1-\lambda)|L|^{\frac{1}{n}}
\end{equation} for $K, L$ nonempty.
See an extensive survey by Gardner \cite{Gar} on the subject for more information. 

B\"or\"oczky, Lutwak, Yang, Zhang \cite{BLYZ} conjectured that a stronger inequality, called Log-Brunn-Minkowski inequality, holds in the case when $K$ and $L$ are symmetric convex sets:
\begin{equation}\label{logbm-leb}
|\lambda K+_0(1-\lambda)L|\geq |K|^{\lambda}|L|^{1-\lambda},
\end{equation}
where the zero-sum stands for 
$$\lambda K+_0(1-\lambda)L:=\{x\in\R^n:\,\forall\,u\in\sfe\,\langle x,u\rangle\leq h_K(u)^{\lambda}h_L(u)^{1-\lambda}\};$$
here the support function of a convex set $K$ is
$$h_K(x):=\sup_{y\in K}\langle x,y\rangle.$$

B\"or\"oczky, Lutwak, Yang, Zhang \cite{BLYZ} verified this conjecture for planar symmetric convex sets. Saraglou \cite{christos1} and Cordero-Erasquin, Fradelizi and Maurey \cite{CFM} verified the conjecture for unconditional convex sets in $\R^n$. Rotem \cite{liran} verified the conjecture for complex convex bodies.

Saraglou \cite{christos} showed that in case the Log-Brunn-Minkowski conjecture holds for Lebesgue measure on $\R^n,$ then it is also correct for any even log-concave measure $\mu$ in $\R^n$: for all symmetric convex sets $K$ and $L$ and any $\lambda\in [0,1]$,
\begin{equation}\label{logbm-meas}
\mu(\lambda K+_0(1-\lambda)L)\geq \mu(K)^{\lambda}\mu(L)^{1-\lambda}.
\end{equation}

More generally, for $p\in [0,1]$, the $L_p$-sum of convex sets is defined as
$$\lambda K+_p(1-\lambda)L:=\{x\in\R^n:\,\forall\,u\in\sfe\,\langle x,u\rangle\leq (\lambda h_K(u)^{p}+(1-\lambda)h_L(u)^{p})^{\frac{1}{p}}\}.$$
The limiting case $p=0$ corresponds to the zero-sum, and the case $p=1$ corresponds to the usual Minkowski sum. The $L_p$-Brunn-Minkowski conjecture states that for all symmetric convex sets $K$ and $L$ and any $\lambda\in [0,1]$, and for $p\in[0,1]$
\begin{equation}\label{pbm-leb}
|\lambda K+_p(1-\lambda)L|\geq |K|^{\lambda}|L|^{1-\lambda}.
\end{equation}
Equivalently (by homogeneity),
\begin{equation}\label{pbm-leb-add}
|\lambda K+_p(1-\lambda)L|^{\frac{p}{n}}\geq \lambda |K|^{\frac{p}{n}}+(1-\lambda)|L|^{\frac{p}{n}}.
\end{equation}
See Remark \ref{remark-homo} for more details. Kolesnikov and Milman \cite{KolMilsupernew}, in conjunction with later results of Chen, Huang, Li, Liu \cite{global} and Putterman \cite{Put} showed that (\ref{pbm-leb}) is true for $p\in [1-cn^{-3/2},1]$.

One of our results is the following

\begin{theorem}\label{lebesgue}
Let $K$ and $L$ be symmetric convex sets in $\R^n$ such that $K\subset L$. Suppose that
$$p\geq 1-{C}{n^{-0.75}},$$
for a sufficiently small absolute constant $C>0.$ Then for any $\lambda>0$,
\begin{equation}\label{mainresult-sect10}
|\lambda K+_p(1-\lambda)L|\geq |K|^{\lambda}|L|^{1-\lambda}.
\end{equation}
\end{theorem}

Note that this improves upon the previous estimate $p\geq 1-{C}{n^{-1.5}}$ of Kolesnikov and Milman \cite{KolMilsupernew}, in the partial case when $K\subset L.$ While we follow the general scheme of \cite{KolMilsupernew}, we find an improvement in this partial case using a different estimate at a certain key step; see Remark \ref{explain} for more details.

\medskip

Independently of B\"or\"oczky, Lutwak, Yang, Zhang (and earlier), Gardner and Zvavitch conjectured \cite{GZ} that for any even log-concave measure $\mu$, any pair of symmetric nonempty convex sets, and any $\lambda\in [0,1]$,  
\begin{equation}\label{GaussBM}
\mu(\lambda K+(1-\lambda)L)^{\frac{1}{n}}\geq \lambda \mu(K)^{\frac{1}{n}}+(1-\lambda)\mu(L)^{\frac{1}{n}}.
\end{equation}
The conjecture cannot hold without any structural assumptions: if, for example, $K=B_2^n$ and $L=B_2^n+Re_1,$ for $R>0$ large enough, the inequality fails. Gardner and Zvavitch \cite{GZ} showed that (\ref{GaussBM}) holds when $K$ and $L$ are dilates of a barycentered convex set, building up on the work of Cordero-Erasquin, Fradelizi and Maurey \cite{CFM}. Nayar and Tkocz \cite{NaTk} showed that the conjecture cannot hold only under the assumption that $K$ and $L$ contain the origin. Kolesnikov and Livshyts \cite{KolLiv} showed that for the Gaussian measure $\mu$ and convex sets $K$ and $L$ containing the origin, the inequality (\ref{GaussBM}) holds with power ${1}/{2n}$ in place of ${1}/{n}$. Livshyts, Marsiglietti, Nayar, Zvavitch \cite{LMNZ} showed that the Log-Brunn-Minkowski conjecture implies the dimensional Brunn-Minkowski conjecture, and thus (\ref{GaussBM}) holds for unconditional convex bodies and for symmetric convex sets on the plane.

In this paper we propose to study the following ``unified'' conjecture (which, as we shall show, follows from the Log-Brunn-Minkowski conjecture):

\begin{conjecture}[the (p,q)-inequality]\label{mainconj}
Fix any $p\in [0,1]$. For any even log-concave measure $\mu$, any pair of nonempty symmetric convex sets $K$ and $L$ and any $\lambda\in [0,1]$, and for any $q\in [0,p]$,
$$\mu(\lambda K+_p (1-\lambda)L)^{\frac{q}{n}}\geq \lambda \mu(K)^{\frac{q}{n}}+(1-\lambda)\mu(L)^{\frac{q}{n}}.$$
\end{conjecture}

Note that 
\begin{itemize}
\item the case $(0,0)$ corresponds to the Log-Brunn-Minkowski conjecture; 
\item the case $(1,0)$ corresponds to Borell's theorem;
\item the case $(p,0)$ corresponds to the $L_p$-Brunn-Minkowski inequality;
\item for Lebesgue measure, $(p,0)$ automatically self-improves to $(p,p)$ by a homogeneity argument. However, this is not the case for a general log-concave measure;
\item the case $(1,1)$ corresponds to the conjecture of Gardner and Zvavitch.
\end{itemize}

It is important to note that for $p\in (0,1],$ this conjecture a-priori does not follow and does not imply the Log-Brunn-Minkowski conjecture. It is also not clear, a-priori, if the validity of this conjecture for $p_1\in [0,1]$ yields the validity of this conjecture for a different $p_2\in [0,1]$. In the case of Lebesgue measure, this implication works for $p_1<p_2.$ In this paper, we shall show that the same implication works for any log-concave measure!

\medskip

We begin by outlining the following implications for the above conjecture, some of which are straight-forward, others go back to previous results, and some we show here.

\begin{remark}
Fix $t>0.$
\begin{enumerate}
\item The $(p,q)$-inequality implies the $(p,q-t)$-inequality, for any fixed pair of $K$ and $L$, fixed $\lambda\in [0,1]$ and a fixed $\mu.$
\item The $(p,q)$-inequality implies the $(p+t,q)$-inequality, for any fixed pair of $K$ and $L$, fixed $\lambda\in [0,1]$ and a fixed $\mu.$
\item (Saraglou) The $(p,0)$ inequality for Lebesgue measure (for all symmetric convex $K,L$) implies the $(p,0)$ inequality for all even log-concave measures $\mu$, and all symmetric convex $K,L$.
\end{enumerate}
Indeed, part (1) follows from H\"older's inequality (the fact that $(\lambda a^p+(1-\lambda)b^p)^{\frac{1}{p}}$ is increasing in $p$), and part (2) follows from the inclusion
$$\lambda K+_p(1-\lambda)L\subset \lambda K+_{p'}(1-\lambda)L,$$
whenever $p\leq p'.$ This inclusion, in turn, also follows from H\"older's inequality. Part (3) was shown by Saraglou in Section 3 of \cite{christos} for $p=0$, and the same argument yields this fact for any $p\in [0,1]$. We would like to note that Saraglou \cite[Section 5 ]{christos} also showed that the inequality (\ref{logbm-meas}), verified \emph{in all dimensions}, say, for the Gaussian measure, implies that it holds for all other log-concave even measures as well, in all dimensions.
\end{remark}

Here, we show, furthermore,

\begin{proposition}[Implication]\label{implications}
Fix $t>0.$ The $(p,q)$-inequality for a fixed measure $\mu$ (for all symmetric convex $K,L$) implies the $(p+t,q+t)$-inequality for $\mu$ and all symmetric convex $K,L.$ In particular, the validity of the Log-Brunn-Minkowski conjecture would imply the validity of Conjecture \ref{mainconj} for all $0\leq q\leq p\leq 1$.
\end{proposition}

Proposition \ref{implications}, in the case when $p=0,$ was verified in \cite{LMNZ}. Here we shall show this more general fact, via an alternative argument, in Section 4. 

\medskip
\medskip

We verify the $(p,q)-inequality$ in certain cases for some range of $p$ and $q$, and our estimates depend on the appropriate parameters of the measure and on the inradius of the sets $K$ and $L:$

\begin{theorem}\label{main}
Let $K$ be a nonempty symmetric convex set in $\R^n$ containing $r B_2^n$. Let $\mu$ be the measure with twice-differentiable density $e^{-V}$, where $V$ is an even convex function, such that
$$\nabla^2 V\geq k_1 Id$$
and 
$$\int_K \Delta V\leq k_2 n\mu(K)$$ for some nonnegative $k_1, k_2$.  
Then for any $\lambda>0$, and a nonempty symmetric convex set $L$ such that $r B_2^n\subset L$, we have
\begin{equation}\label{conclusion}
\mu(\lambda K+_p(1-\lambda)L)^{\frac{q}{n}}\geq \lambda\mu(K)^{\frac{q}{n}}+(1-\lambda)\mu(L)^{\frac{q}{n}},
\end{equation}
whenever 
\begin{itemize}
\item $k_1\in [\frac{1}{n},1]$ and
$$(1-p)\frac{1+n}{r^2}+q(1+k_2)(1+k_1)\leq 2k_1;$$
\item Alternatively, for all $k_1\geq 0,$ (\ref{conclusion}) follows whenever
$$\begin{cases}
q\left(k_1^2+k_1k_2-(nk_1+k_2)\frac{1-p}{r^2}\right)\leq k_1^2+n\left(\frac{1-p}{r^2}\right)^2-\frac{1-p}{r^2}(n+1)k_1;\\
\frac{1-p}{r^2}\leq \frac{k_1}{n}.
\end{cases}$$
\end{itemize}
\end{theorem}

Additionally, we show

\begin{proposition}\label{prop-main}
Under the assumptions of Theorem \ref{main}, assuming additionally that $K\subset L$, and assuming that $k_1\leq 1$, we moreover get the conclusion (\ref{conclusion}) with the assumption 
$$(1-p)\frac{2\sqrt{n}\sqrt{1+k_2}\sqrt{1+k_1}+\sqrt{k_1}}{2r}+q(1+k_2)(1+k_1)\leq 2k_1.$$ 
\end{proposition}

As a corollary of Theorem \ref{main} and Proposition \ref{prop-main}, we get a result when $\mu$ is the $n$-dimensional Gaussian measure, defined to have a density $(2\pi)^{-n/2} e^{-|x|^2/2} dx $, as in this case $k_1=k_2=1$. 

\begin{cor}\label{thm-gauss}
Let $\gamma$ be the Gaussian measure, and let $K$ and $L$ be symmetric convex sets containing $rB^n_2$. Then for any $\lambda>0$,
\begin{enumerate}
\item $\gamma(\lambda K+_p(1-\lambda)L)\geq \gamma(K)^{\lambda}\gamma(L)^{1-\lambda},$ whenever $p\geq 0$ and
$$p\geq 1-\frac{2r^2}{n+1}.$$
\item In particular, the Gaussian Log-Brunn-Minkowski inequality holds for all convex sets $K$ and $L$ containing $\sqrt{0.5(n+1)}B_2^n.$
\item More generally, $\gamma(\lambda K+_p(1-\lambda)L)^{\frac{q}{n}}\geq \lambda\gamma(K)^{\frac{q}{n}}+(1-\lambda)\gamma(L)^{\frac{q}{n}},$ provided that
$$4q+\frac{n+1}{r^2}(1-p)\leq 2.$$
\item Assuming further that $K\subset L$, we show that $\gamma(\lambda K+_p(1-\lambda)L)\geq \gamma(K)^{\lambda}\gamma(L)^{1-\lambda},$ whenever $p\geq 0$ and
$$p\geq 1-\frac{r}{\sqrt{n}+0.25}.$$
\end{enumerate}
\end{cor}

Note that part (3) implies some of the results from \cite{KolLiv}, corresponding to the case $p=1$. 

\medskip

In addition to the above, we verify the $(p,p)$-inequality for all $p\in [0,1]$, in the partial case when $K$ and $L$ are dilates. The result below extends both the B-theorem of Cordero, Fradelizi and Maurey \cite{CFM} and a result of Gardner and Zvavitch \cite{GZ}.

\begin{theorem}\label{dilates}
Conjecture \ref{mainconj} holds in the case when the measure is Gaussian and $K$ and $L$ are dilates of each other. That is, for any convex set $K$ and any $t\in\R$, and for all $p\in [0,1],$ $\lambda\in[0,1],$
$$\gamma(\lambda K+_p(1-\lambda)tK)^{\frac{p}{n}}\geq \lambda \gamma(K)^{\frac{p}{n}}+(1-\lambda)\gamma(tK)^{\frac{p}{n}}.$$
\end{theorem}

The methods of our proof involve considering \emph{local versions} of the aforementioned functional inequalities, building up on the methods developed by Kolesnikov and Milman \cite{KM1}, \cite{KM2}, \cite{KolMil}, \cite{KolMilsupernew}, Colesanti, Livshyts, Marsiglietti \cite{Col1}, \cite{CLM}, \cite{CL}, Kolesnikov and Livshyts \cite{KolLiv}. In particular, we use a Bochner-type identity obtained in \cite{KM1}.

In Section 2 we derive local versions of the inequalities. In Section 3 we show that the local version implies the global version, \emph{for any fixed measure $\mu$}, using the method of Putterman \cite{Put} (whose result was derived in the Lebesgue case). In Section 4 we show the Proposition \ref{implications}. In Section 5 we describe a reduction of the inequality using integration by parts. In Section 6 we do several preparatory estimates. In Section 7 we show the proof of Theorem \ref{lebesgue}. In Section 8 we verify Theorem \ref{main}. In Section 9 we verify Proposition \ref{prop-main}. In Section 10 we prove Theorem \ref{dilates}.

\textbf{Acknowledgement.} The second named author was supported by RFBR project 20-01-00432; the second named author has been funded by the Russian Academic Excellence Project ``5-100''. The third named author is supported by the NSF CAREER DMS-1753260. The authors are very grateful to the anonymous referee for a number of suggestions of crucial importance, thanks to which the paper got improved a lot.

\section{Infinitesimal forms}

Below, an even log-concave measure $\mu$ with density $e^{-V}$ on $\R^n$ is fixed, and we assume that $V\in C^2(\R^n).$ Given a convex set $K$, $\rm{II}$ stands for its quadratic form, and $H_x$ is the weighted mean curvature at $x$ associated with the measure $\mu$:
$$H_x= tr({\rm II})-\langle \nabla V, n_x\rangle.$$

In order to derive our results, we reduce the problem to its infinitesimal version following the approach of \cite{Col1}, \cite{CLM}, \cite{CL}, \cite{KM1}, \cite{KM2}, \cite{KolMil}, \cite{KolMilsupernew}. 

\begin{lemma}[the infinitesimal form of the (p,q)-inequality \ref{mainconj}]\label{local}
Suppose Conjecture \ref{mainconj} holds for the measure $\mu$ with the parameters $p$ and $q$. Then for any $C^{2,+}$ symmetric convex set $K$, and for any twice-differentiable $f:\partial K\rightarrow\R$, we have 
\begin{equation}\label{local-key-eq}
\int_{\partial K} H_x f^2-\langle \mbox{\rm{II}}^{-1}\nabla_{\partial K}  f,  \nabla_{\partial K}  f\rangle +(1-p)\frac{f^2}{\langle x,n_x\rangle}d\mu-\frac{n-q}{n\mu(K)}\left(\int_{\partial K} f d\mu  \right)^2\leq 0.
\end{equation}
\end{lemma}
\begin{proof} We apply the argument ``the global concavity implies the local concavity''. More precisely, we use the following fact:  the  inequality
$$\mu(\lambda K+_p (1-\lambda)L)^{\frac{q}{n}}\geq \lambda \mu(K)^{\frac{q}{n}}+(1-\lambda)\mu(L)^{\frac{q}{n}}$$
implies
\begin{equation}
\label{qndd}
\frac{d^2}{d \varepsilon^2} \mu(K +_p \varepsilon f)^{\frac{q}{n}} \le 0.
\end{equation}
for sufficiently regular $f$ on the unit sphere  and strictly convex $K$ with $C^2$-boundary, where 
$
K +_p \varepsilon f
$
is a convex body with support function
$
\sqrt[p]{h^p_K + \varepsilon f^p}.
$
Note that for sufficiently small values of $\varepsilon$, this function is indeed a support function, given that $K$ is strictly convex. The proof mimics the arguments of Lemma 3.4 in \cite{KolMilsupernew} and we omit it here.

We use the second-order Taylor expansion
$$
h_K +_p \varepsilon f = h_K + \varepsilon z h_K + \frac{\varepsilon^2}{2} (1-p) z^2 h_K + o(\varepsilon),
$$
where 
$$
z = \frac{f^p}{p h^p_K}, \ {\rm if} \  p \ne 0,
$$
and
$$
z = \log f, \ {\rm if} \  p=0.
$$
We recall the expressions for derivatives of $\mu(K + \varepsilon f)$ from \cite{KM2}:
\begin{equation}\label{1der}
\frac{d}{d \varepsilon} \mu(K + \varepsilon f) |_{\varepsilon=0} = \int_{\partial K} g  d\mu
\end{equation}
and
\begin{equation}\label{2der}
\frac{d^2}{d \varepsilon^2} \mu(K + \varepsilon f) |_{\varepsilon=0} 
=
\int_{{\partial K}} H_x g^2 d \mu - \int_{{\partial K}} \langle II^{-1} \nabla_{\partial K} g, \nabla_{\partial K} g \rangle d \mu,
\end{equation}
where 
$$
g = f(n_x).
$$
We will use later the fact that the first derivative identity (\ref{1der}) does not require any regularity assumption on $K$. A proof is provided in the appendix. 

Applying the Taylor expansion along with these formulas, we get
$$
\frac{d}{d \varepsilon} \mu(K +_p \varepsilon f) |_{\varepsilon=0}= \int_{\partial K} f d\mu
$$
\begin{align*}
\frac{d^2}{d \varepsilon^2} \mu(K +_p \varepsilon f) |_{\varepsilon=0} 
& =
\int_{{\partial K}} H_x f^2 d \mu - \int_{{\partial K}} \langle II^{-1} \nabla_{{\partial K}} f, \nabla_{{\partial K}} f \rangle d \mu
+ (1-p) \int_{{\partial K}} \frac{f^2}{h_K(n_x)} d\mu,
\end{align*}
where 
$$
f = (z h_K)(n_x). 
$$
Thus (\ref{qndd}) reads as
$$
\int_{{\partial K}} H_x f^2 d \mu - \int_{{\partial K}} \langle II^{-1} \nabla_{{\partial K}} f, \nabla_{{\partial K}} f \rangle d \mu
+ (1-p) \int_{{\partial K}} \frac{f^2}{h_K(n_x)} d\mu
 \le \frac{1}{\mu(K)} \Bigl( 1- \frac{q}{n}\Bigr) \Bigl( \int_{\partial K} f d \mu \Bigr)^2.
$$
\end{proof}

We note that Kolesnikov and Milman \cite{KolMilsupernew} showed that the inequality (\ref{local-key-eq}) is true when $K=B^n_p$, for $p\in [2,\infty]$, provided that $n>c(p)$. For $p=2$ this was also verified by Colesanti, Livshyts and Marsiglietti \cite{CLM}.

\begin{remark}\label{remark-homo}
Consider an arbitrary symmetric bilinear form $Q: A\times A\rightarrow \R$, where $A$ is a linear space. Suppose that for every $a\in A,$ one has
\begin{equation}\label{ineq-Q}
Q(a,a)\leq 0. 
\end{equation}
Fix any element $z\in A.$ We note that one may always improve (\ref{ineq-Q}) and make it invariant under scaled addition of $z$. Indeed, (\ref{ineq-Q}) implies that for every $t\in\R$,
\begin{equation}\label{ineq-Q-t}
Q(a+tz,a+tz)=Q(a,a)+2tQ(a,z)+t^2Q(z,z)\leq 0. 
\end{equation}
Viewing (\ref{ineq-Q-t}) as a family of inequalities indexed by $t\in\R$, we note that (\ref{ineq-Q-t}) is sharpest possible when $t=-\frac{Q(a,z)}{Q(z,z)}$, and in this case it becomes
\begin{equation}\label{ineq-sh}
Q(a,a)\leq \frac{Q(a,z)^2}{Q(z,z)}. 
\end{equation}
Note that (\ref{ineq-sh}) is sharper than (\ref{ineq-Q-t}), and, importantly, the inequality (\ref{ineq-sh}) is invariant under the change $a\rightarrow a+sz,$ for any $s\in\R.$

\medskip

We apply this abstract observation with the algebra $A$ of smooth functions on $\partial K,$ the bilinear form
$$Q(f,g)=\int_{\partial K} H_x fg-\langle \mbox{\rm{II}}^{-1}\nabla_{\partial K}  f,  \nabla_{\partial K}  g\rangle +(1-p)\frac{fg}{\langle x,n_x\rangle}d\lambda-\frac{1}{|K|}\left(\int_{\partial K} f d\lambda  \right)\left(\int_{\partial K} g d\lambda  \right),$$
where $\lambda$ is Lebesgue measure, and the special function $z(x)=\langle x,n_x\rangle$. Integration by parts yields that the inequality (\ref{local-key-eq}) with $\mu=\lambda$ automatically yields the inequality
\begin{equation}\label{improved}
\int_{\partial K} H_x f^2-\langle \mbox{\rm{II}}^{-1}\nabla_{\partial K}  f,  \nabla_{\partial K}  f\rangle +(1-p)\frac{f^2}{\langle x,n_x\rangle}d\lambda-\frac{n-p}{n|K|}\left(\int_{\partial K} f d\lambda  \right)^2\leq 0,
\end{equation}
as per the argument above, according to which, generally, (\ref{ineq-Q}) yields an a-priori stronger inequality (\ref{ineq-sh}). 

This (together with the local-to-global result of Putterman \cite{Put}, and with Lemma \ref{local}) explains why (\ref{pbm-leb}) is equivalent to (\ref{pbm-leb-add}), and not just weaker. Alternatively, a standard elementary argument can show this fact as well. 

The underlying reason why the ``improved'' inequality (\ref{improved}) assumes such a nice form is the homogeneity of Lebesgue measure. The choice of function $z=\langle x,n_x\rangle$ corresponds, geometrically, to taking additional dilates of $K$.

An important feature of (\ref{improved}) is its invariance under the change $f\rightarrow f+t\langle x,n_x\rangle$, which was previously noticed and used by Kolesnikov and Milman \cite{KolMilsupernew}.
\end{remark}

\section{Local implies global}


In this section, we show that verifying the local form of the $(p,q)$ inequality leads to the global form. We will use methods developed by Putterman \cite{Put}.

We begin by recalling various notations and definitions. Let $f$ be a positive continuous function on the sphere. The Wulff shape of $f$ is the set \begin{align*}
    W(f) &= \{x \in \mathbb{R}^n: \forall \ u \in S^{n-1} \ \langle x, u\rangle \le f(u)\}.
\end{align*} 
See, e.g. \cite{BLYZ}, \cite{BLYZ-1}, \cite{BLYZ-2} or \cite{book4} for a discussion and properties of Wulff shapes. Observe that $W(f)$ is the intersection of closed half-spaces containing the origin and is therefore a convex body. We shall use notation
$$(1-\lambda)h_K +_p \lambda h_L=((1-\lambda)h_K^p + \lambda h_L^p)^{\frac{1}{p}}.$$
Recall also that \begin{align*}
    (1-\lambda)K +_p \lambda L &= W((1-\lambda)h_K +_p \lambda h_L).
\end{align*} 

In this section, we will use $g$ to denote the continuous density of our log-concave even measure $\mu$. Given a convex body $K$ with Gauss map $\nu_K$, the surface area measure of $K$ with respect to $\mu$ is defined as \begin{align*}
    \sigma_{\mu, K}(\Omega) &= \int_{\nu_K^{-1}(\Omega)}g(x) dH_{n-1}(x) 
\end{align*} for all Borel $\Omega \subset S^{n-1}$. Here $H_{n-1}$ stands for the $(n-1)-$dimensional Hausdorff measure. Observe that in the special case when $K$ is a polytope with outer normals $u_i, 1\le i\le N$ and corresponding faces $F_i, 1\le i\le N$ we have \begin{equation}\label{psurf}
    d\sigma_{\mu, K}(u) = \sum_{i=1}^{N}\delta_{u_i}\mu_{n-1}(F_i)du,
\end{equation} where $\mu_{n-1}(F_i) := \int_{F_i}g(x)dH_{n-1}(x)$.

Let us now consider a family $\mathcal{F}$ of symmetric convex sets that is closed under $L_p-$Minkowski convex interpolation and that is open with respect to the Hausdorff metric. This means that for every $K, L \in \mathcal{F}$ and $\lambda \in [0,1], p\ge 0$ we have $(1-\lambda)K +_p \lambda L \in \mathcal{F}$ and that for every $K \in \mathcal{F}$ there exists $\varepsilon > 0$ such that $d(K, L) < \varepsilon$ for a symmetric convex body $L$ implies that $L \in \mathcal{F}$. 

\begin{theorem}\label{loc->glob}
Assume that (\ref{local-key-eq}) holds for some $p, q < 1$ for any $C^{2,+}$ symmetric convex $K \in \mathcal{F}$ and any even $C^1$-smooth $f:\partial K\rightarrow \R$, that is
$$\int_{\partial K} H_x f^2-\langle \mbox{\rm{II}}^{-1}\nabla_{\partial K}  f,  \nabla_{\partial K}  f\rangle +(1-p)\frac{f^2}{\langle x,n_x\rangle}d\mu-\frac{n-q}{n\mu(K)}\left(\int_{\partial K} f d\mu  \right)^2\leq 0.$$
Then, for any symmetric convex sets $K, L\in \mathcal{F}$ and $\lambda \in [0,1]$ we have \begin{align*}
    \mu((1-\lambda) K +_p \lambda L)^{\frac{q}{n}} \ge (1-\lambda) \mu(K)^{\frac{q}{n}} + \lambda \mu(L)^{\frac{q}{n}}.
\end{align*}
\end{theorem}

As a prototypical example, one can take $\mathcal{F}$ to be the set of symmetric convex bodies that contain a ball of a given radius. In addition, Theorem 3.1 can also be applied in the case when $\mathcal{F}$ is simply the set of symmetric convex bodies. 




Our proof will be accomplished through approximation by strongly isomorphic polytopes. Let us recall:
\begin{definition}
Two polytopes $K$ and $L$ are said to be strongly isomorphic if \begin{align*}
\dim F(K, u) = \dim F(L, u)
\end{align*} for all $u \in S^{n-1}$, where $F(K, u)$ denotes the support set $\{x \in K: \langle x, u\rangle = h_K(u)\}$.
\end{definition}

When $K, L$ and $p$ are all assumed fixed, let us employ the notation 
$$K_{\lambda} = (1-\lambda) K +_p \lambda L.$$

We use the following lemma:
\begin{lemma}
Let $\alpha, \beta, \gamma \in [0,1]$ and $\lambda = (1-\gamma)\alpha + \gamma \beta$. If $K_{\alpha}, K_{\lambda},$ and $K_{\beta}$ are strongly isomorphic then $K_{\lambda} = (1-\gamma)K_{\alpha} +_p  \gamma K_{\beta}$.
\end{lemma}
\begin{proof}
Let $u_1, ..., u_N$ be the facet normals of $K_{\alpha}, K_{\lambda},$ and $K_{\beta}$. We may write 
\begin{align*}
K_{\lambda} &= \{x : \langle x, u_i \rangle \le ( (1-\lambda)h_K +_p \lambda h_L)(u_i), i = 1, ..., N\}.
\end{align*} 
Since each $u_i$ is also a facet normal of $K_{\alpha}$ and $K_{\beta}$, we have 
$$h_{K_{\alpha}}(u_i) = ((1-\alpha)h_K+_p \alpha h_L)(u_i)$$ 
and 
$$h_{K_{\beta}}(u_i) = ((1-\beta)h_K +_p \beta h_L)(u_i).$$ 

Therefore, \begin{align*}
((1-\gamma)h_{K_{\alpha}} +_p \gamma h_{K_{\beta}})(u_i) &= ((1-\gamma)((1-\alpha)h_K+_p \alpha h_L) +_p \gamma((1-\beta)h_K +_p \beta h_L))(u_i) \\ &= ((1-\lambda)h_K +_p \lambda h_L)(u_i),
\end{align*} and so our proof is concluded.
\end{proof}

A further ingredient we need is a weak-convergence result for the surface area measure of a convex body with respect to $\mu$. 
\begin{lemma}\label{hausdorffsurfacearea}
Let $K, L$ be convex bodies in $\mathbb{R}^n$ within Hausdorff distance $\varepsilon$ from each other, $\varepsilon>0$. Then for every bounded function $a(u)$, \begin{align*}
    \left|\int_{S^{n-1}}a(u)d\sigma_{\mu, K}(u) - \int_{S^{n-1}}a(u)d\sigma_{\mu, L}(u) \right| \le C \varepsilon,
\end{align*} where the constant $C>0$ depends on $\|a\|_{\infty}, g, \max_{x \in K}|x|, \max_{x\in L}|x|.$ 
\end{lemma}

In the case of Lebesgue measure, this lemma is implicit in results proved in Schneider \cite{book4}. For the general case, one can argue by approximating $K, L$ by strongly isomorphic polytopes. To compare the $(n-1)-$dimensional $\mu$ measure of corresponding faces, one uses the fact that $g$ is Lipschitz on compact sets.


\begin{lemma}
Assume that (\ref{local-key-eq}) holds for some $p, q < 1$ for any $C^{2,+}$ symmetric convex $K \in \mathcal{F}$ and any even $C^1$-smooth $f:\partial K\rightarrow \R$. Then we have the following statement: For any two strongly isomorphic symmetric polytopes $K, L \in \mathcal{F}$, and for any $\lambda \in [0,1]$ such that there exists a (possibly one-sided) neighborhood $U$ of $\lambda$ for which all the $\{ K_{\lambda'}: \lambda' \in U\}$ are strongly isomorphic to one another, we have $\frac{d^2}{d\lambda^2}\left( \mu(K_{\lambda})^{\frac{q}{n}}\right)|_{\lambda} \le 0$.
\end{lemma}
\begin{proof}
By Lemma 3.3 and the fact that $\mathcal{F}$ is closed under Minkowski interpolation, we may reduce our problem to showing that if $K, L$ are strongly isomorphic polytopes such that $K_{\lambda}$ is strongly isomorphic to $K$ and $L$ for all $\lambda \in [0,1]$ then we have $\frac{d^2}{d\lambda^2}(\mu(K_{\lambda})^{\frac{q}{n}}) \big|_{\lambda=0} \le 0$. 

Let $u_1, ..., u_N$ denote the set of outer normals to $K$ (and all $K_{\lambda}$) and let $F_i(K_{\lambda})$ denote the face of $K_{\lambda}$ with outer normal $u_i$.

Let $h_i = h_K(u_i)$. We choose $s_i$ such that $h_L(u_i) = h_i(1+ps_i)^{\frac{1}{p}}$ and define \begin{align*}
    &a_i(\lambda) = (1+\lambda p s_i)^{\frac{1}{p}}, \\
    &b_i(\lambda) = (1+\lambda p s_i)^{\frac{1-p}{p}}, \\
    &c_i(\lambda) = (1+\lambda p s_i)^{\frac{1-2p}{p}}.
\end{align*} Observe that $h_{K_{\lambda}}(u_i) = h_i a_i(\lambda)$. Since $K_{\lambda}$ is strongly isomorphic to $K$, its surface area measure is given by $\sum_{i=1}^{N} \mu_{n-1}(F_i(K_{\lambda})) \delta_{u_i}$ by (15). 

We have \begin{align}\label{firstder}\begin{split}
    \frac{d}{d\lambda} \mu(K_{\lambda}) &= \lim_{\varepsilon \to 0} \frac{1}{\varepsilon}\left(\mu(K_{\lambda+\varepsilon}) - \mu(K_{\lambda})\right) \\ &= \lim_{\varepsilon \to 0} \frac{1}{\varepsilon} \sum_{i=1}^{N} \int_{h_i a_i(\lambda)}^{h_i a_i(\lambda + \varepsilon)} \mu_{n-1}(F_i(K_{(h_ia_i)^{-1}(\alpha)})) d\alpha \\ &= \sum_{i=1}^{N} \frac{d}{d\lambda}(h_ia_i(\lambda)) \mu_{n-1}(F_i(K_{\lambda})) \\ &= \sum_{i=1}^{N} h_i s_i b_i(\lambda) \mu_{n-1}(F_i(K_{\lambda})).
    \end{split}
\end{align} Hence, \begin{align*}
    \frac{d}{d\lambda}\mu(K_{\lambda})\big|_{\lambda=0} &= \sum_{i=1}^{N} h_i s_i \mu_{n-1}(F_i(K)).
\end{align*} We also compute, by the product rule applied to (\ref{firstder}), \begin{align*}
    \frac{d^2}{d\lambda^2} \mu(K_{\lambda}) &= (1-p)\sum_{i=1}^{N} s_i^2 h_i c_i(\lambda) \mu_{n-1}(F_i(K_{\lambda})) + \sum_{i=1}^{N}h_i s_i b_i(\lambda) \frac{d}{d\lambda} \mu_{n-1}(F_i(K_{\lambda})),
\end{align*} and so \begin{align*}
    \frac{d^2}{d\lambda^2} \mu(K_{\lambda})\big|_{\lambda = 0} &= (1-p)\sum_{i=1}^{N}s_i^2 h_i \mu_{n-1}(F_i(K)) + \sum_{i=1}^{N} h_i s_i \frac{d}{d\lambda} \mu_{n-1}(F_i(K_{\lambda}))\big|_{\lambda = 0}.
\end{align*} The claim that $\frac{d^2}{d\lambda^2}\left(\mu(K_{\lambda})^{\frac{q}{n}}\right) \le 0$ is equivalent to the claim that 
$$\frac{d^2}{d\lambda^2}\mu(K_{\lambda}) \big|_{\lambda = 0} - \frac{n-q}{n\mu(K)}\left(\frac{d}{d\lambda} \mu(K_{\lambda})\big|_{\lambda = 0} \right)^2 \le 0.$$ 
Using the above expressions, we thus wish to demonstrate that 
$$
(1-p)\sum_{i=1}^{N}s_i^2 h_i \mu_{n-1}(F_i(K)) + \sum_{i=1}^{N} h_i s_i \frac{d}{d\lambda} \mu_{n-1}(F_i(K_{\lambda}))\big|_{\lambda = 0}$$$$\le \frac{n-q}{n\mu(K)}\left(\sum_{i=1}^{N}h_i s_i \mu_{n-1}(F_i(K)) \right)^2.
$$

Let $K_{\varepsilon}, L_{\varepsilon} \in C^{2,+}_{e}$ be approximations to $K, L$ respectively such that 
$$d(K_{\varepsilon}, K), d(L_{\varepsilon}, L) < \varepsilon$$ 
in the Hausdorff metric. For sufficiently small $\varepsilon$, we have $K_{\varepsilon}, L_{\varepsilon} \in \mathcal{F}$. Therefore the inequality (\ref{local-key-eq}) yields
\begin{align}\label{approxlocal}
(1-p)\int_{\partial K_{\varepsilon}} \frac{f^2}{\langle x, n_x\rangle} d\mu_{\partial K_{\varepsilon}}(x) + \int_{\partial K_{\varepsilon}}H_x f^2-\langle \mbox{\rm{II}}^{-1}\nabla_{\partial K_{\varepsilon}}  f,  \nabla_{\partial K_{\varepsilon}}  f\rangle d\mu_{\partial K_{\varepsilon}} &\le 
\end{align} 
$$\frac{n-q}{n\mu(K_{\varepsilon})}\left(\int_{\partial K_{\varepsilon}} f d\mu_{\partial K_{\varepsilon}} \right)^2$$
for any $C^1-$smooth $f: \partial K_{\varepsilon} \to \mathbb{R}$. Let us take $f_{\varepsilon}(x) = w_{\varepsilon}(n_x)$ where \begin{align}\label{wepsilon} w_{\varepsilon}(u) = \frac{1}{p}h_{K_{\varepsilon}}(u) \left(\left(\frac{h_{L_{\varepsilon}}(u)}{h_{K_{\varepsilon}}(u)} \right)^p - 1 \right).\end{align} Note that $w_0(u_i) = h_i s_i.$ By Lemma 3.4, the fact that $w_{\varepsilon} \to w$ uniformly on $S^{n-1}$, we have \begin{align*}
    \int_{\partial K_{\varepsilon}} f_{\varepsilon} d\mu_{\partial K_{\varepsilon}} &= \int_{S^{n-1}}w_{\varepsilon}(\theta) d\sigma_{\mu, {K_{\varepsilon}}}(\theta) \\ &\to \int_{S^{n-1}}w(\theta) d\sigma_{\mu, K}(\theta) \\
    &= \sum_{i=1}^{N}h_i s_i \mu_{n-1}(F_i(K)).
\end{align*} Similarly, \begin{align*}
\int_{\partial K_{\varepsilon}} \frac{f_{\varepsilon}^2}{\langle x, n_x\rangle} d\mu_{\partial K_{\varepsilon}}(x) &= \int_{S^{n-1}} \frac{w_{\varepsilon}^2(\theta)}{h_{K_{\varepsilon}}(\theta)} d\sigma_{\mu, K_{\varepsilon}}(\theta) \\ &\to \int_{S^{n-1}}\frac{w^2(\theta)}{h_K(\theta)} d\sigma_{\mu, K}(\theta)
\\ &= \sum_{i=1}^{N}s_i^2 h_i \mu_{n-1}(F_i(K)).
\end{align*} Furthermore, by the second derivative formula (\ref{2der}): \begin{align*}
    \int_{\partial K_{\varepsilon}} H_x f_{\varepsilon}^2 - \langle \mbox{\rm{II}}^{-1} \nabla_{\partial K_{\varepsilon}}f_{\varepsilon}, \nabla_{\partial K_{\varepsilon}} f_{\varepsilon} \rangle d\mu_{\partial K_{\varepsilon}} &= \frac{d^2}{ds^2}\mu(W(h_{K_{\varepsilon}} + s w_{\varepsilon}))\big|_{s=0}.
\end{align*} 

It therefore remains to show that\begin{align}\label{secderivlimit}
\lim_{\varepsilon \to 0} \frac{d^2}{ds^2} \mu(W(h_{K_{\varepsilon}} + sw_{\varepsilon}))\big|_{s=0} &= \sum_{i=1}^{N}h_i s_i \frac{d}{d\lambda} \mu_{n-1}(F_i(K_{\lambda})\big|_{\lambda = 0}.
\end{align}

By Lemma 11.1, we know that \begin{align*}
    \frac{d}{ds} \mu(W(h_{K_{\varepsilon}}+sw_{\varepsilon}))\big|_{s=s_0} &= \int_{S^{n-1}} w_{\varepsilon}(\theta) d\sigma_{\mu, W(h_{K_{\varepsilon}} + s_0 w_{\varepsilon})}(\theta).
\end{align*}

Therefore \begin{align*}
    &\frac{d^2}{ds^2}\mu(W(h_{K_{\varepsilon}} + sw_{\varepsilon}))\big|_{s=0} \\ &= \lim_{s\to 0}\frac{1}{s}\left( \int_{S^{n-1}} w_{\varepsilon}(\theta)g(\nu_{W(h_{K_{\varepsilon}} + s w_{\varepsilon})}^{-1}(\theta)) d\sigma_{W(h_{K_{\varepsilon}} + sw_{\varepsilon})}(\theta) - \int_{S^{n-1}}w_{\varepsilon}(\theta) g(\nu_{K_{\varepsilon}}^{-1}(\theta)) d\sigma_{K_{\varepsilon}}(\theta) \right),
\end{align*} and so \begin{align*}
    &\lim_{\varepsilon \to 0} \frac{d^2}{ds^2} \mu(W(h_{K_{\varepsilon}} + sw_{\varepsilon}))|_{s= 0} \\ &= \lim_{\varepsilon\to 0}\lim_{s \to 0}\frac{1}{s}\left( \int_{S^{n-1}} w_{\varepsilon}(\theta)g(\nu_{W(h_{K_{\varepsilon}} + s w_{\varepsilon})}^{-1}(\theta)) d\sigma_{W(h_{K_{\varepsilon}} + sw_{\varepsilon})}(\theta) - \int_{S^{n-1}}w_{\varepsilon}(\theta) g(\nu_{K_{\varepsilon}}^{-1}(\theta)) d\sigma_{K_{\varepsilon}}(\theta) \right).
\end{align*} We denote the limiting expression by $\Pi(s,\varepsilon)$.

We now show that we can interchange the limits in $s$ and $\varepsilon$. This will rely on the Moore-Osgood theorem. It is clear by Hausdorff continuity and Lemma \ref{hausdorffsurfacearea} that $\lim_{\varepsilon \to 0} \Pi(s,\varepsilon)$ always exists for $s\neq 0$. Therefore, it must be shown that $\lim_{s\to 0} \Pi(s,\varepsilon)$ is uniform for $\varepsilon \neq 0$.

We write $\Pi(s,\varepsilon) = \Pi_1(s,\varepsilon) + \Pi_2(s,\varepsilon)$,
 where \begin{align*}
    \Pi_1(s,\varepsilon) &= \frac{1}{s} \int_{S^{n-1}}w_{\varepsilon}(\theta) (g(\nu_{W(h_{K_{\varepsilon}}+sw_{\varepsilon})}^{-1}(\theta))-g(\nu_{K_{\varepsilon}}^{-1}(\theta))) d\sigma_{W(h_{K_{\varepsilon}} + sw_{\varepsilon})}(\theta)
\end{align*} and \begin{align*}
    \Pi_2(s,\varepsilon) &= \frac{1}{s} \int_{S^{n-1}}w_{\varepsilon}(\theta)g(\nu_{K_{\varepsilon}}^{-1}(\theta)) (d\sigma_{W(h_{K_{\varepsilon}} + sw_{\varepsilon})} - d\sigma_{K_{\varepsilon}})(\theta).
\end{align*}

We first show that $\lim_{s\to 0}\Pi_1(s,\varepsilon)$ is uniform for $\varepsilon \neq 0$.

To achieve this, it suffices to show that, for any fixed $\theta \in S^{n-1}$, \begin{align*}
    \lim_{s\to 0}\frac{1}{s}(g(\nu_{W(h_{K_{\varepsilon}}+sw_{\varepsilon})}^{-1}(\theta))-g(\nu_{K_{\varepsilon}}^{-1}(\theta)))
    \end{align*} is uniform for $\varepsilon \neq 0$. Now, \begin{align*}
    \nu_{W(h_{K_{\varepsilon}}+sw_{\varepsilon})}^{-1}(\theta) &= \nabla (h_{K_{\varepsilon}} + sw_{\varepsilon})(\theta) \\ &= \nabla h_{K_{\varepsilon}}(\theta) + s \nabla w_{\varepsilon}(\theta) \\ &= \nu_{K_{\varepsilon}}^{-1}(\theta) + s\nabla w_{\varepsilon}(\theta),
\end{align*} where we have used the fact that $C^2_{h,e}(S^{n-1})$, the space of even $C^2-$support functions, is open in $C^2_{e}(S^{n-1})$, the space of even $C^2-$functions on the sphere, and moreover that $\nu_{K}^{-1} = \nabla h_K$ for $C^{2,+}$ convex bodies $K$. 
Therefore, \begin{align*}
\lim_{s\to 0}\frac{1}{s}(g(\nu_{W(h_{K_{\varepsilon}}+sw_{\varepsilon})}^{-1}(\theta))-g(\nu_{K_{\varepsilon}}^{-1}(\theta))) &= D_{\nabla w_{\varepsilon}(\theta)}g(\nu_{K_{\varepsilon}}^{-1}(\theta))  
\end{align*} and for every $s>0$, there exists $t(s) \in [0,s]$ such that \begin{align*}
    \frac{1}{s}(g(\nu_{W(h_{K_{\varepsilon}}+sw_{\varepsilon})}^{-1}(\theta))-g(\nu_{K_{\varepsilon}}^{-1}(\theta))) &= D_{\nabla w_{\varepsilon}(\theta)}g(\nu_{K_{\varepsilon}^{-1}}(\theta) + t(s) \nabla w_{\varepsilon}(\theta)).
\end{align*} Thus, \begin{align}\begin{split}\label{uniflimit}
    &\left|\frac{1}{s}(g(\nu_{W(h_{K_{\varepsilon}}+sw_{\varepsilon})}^{-1}(\theta))-g(\nu_{K_{\varepsilon}}^{-1}(\theta))) - D_{\nabla w_{\varepsilon}(\theta)} g(\nu_{K_{\varepsilon}}^{-1}(\theta)) \right| \\ &= \left|D_{\nabla w_{\varepsilon}(\theta)}g(\nu_{K_{\varepsilon}}^{-1}(\theta) + t(s) \nabla w_{\varepsilon}(\theta)) - D_{\nabla w_{\varepsilon}(\theta)} g(\nu_{K_{\varepsilon}}^{-1}(\theta)) \right| \\ &\le \left|Dg(\nu_{K_{\varepsilon}}^{-1}(\theta) + t(s) \nabla w_{\varepsilon}(\theta)) - Dg(\nu_{K_{\varepsilon}}^{-1}(\theta)) \right| \left|\nabla w_{\varepsilon}(\theta)\right|.
\end{split}
\end{align} Recalling the definition of $w_{\varepsilon}$ (\ref{wepsilon}), we see that $\nabla w_{\varepsilon}$ can be bounded in terms of $h_{K_{\varepsilon}}, h_{L_{\varepsilon}}, \nabla h_{K_{\varepsilon}}, \nabla h_{L_{\varepsilon}}$. As $h_{K_{\varepsilon}} \le 2 \sup_{x \in K}|x|, h_{L_{\varepsilon}} \le 2\sup_{x \in L}|x|$ and $\nabla h_{K_{\varepsilon}} \in \partial K_{\varepsilon} \subset 2K, \nabla h_{L_{\varepsilon}} \in \partial L_{\varepsilon} \subset 2L$, it follows that there exists a uniform bound (independent of $\varepsilon$) such that \begin{align}\label{unifbound}
    |\nabla w_{\varepsilon}(\theta)| &\lesssim 1.
\end{align} Combining (\ref{uniflimit}), (\ref{unifbound}), and the fact that $Dg$ is uniformly continuous on a compact set, it follows that our desired limit is indeed uniform for $\varepsilon \neq 0$. 

We now prove that $\lim_{s\to 0}\Pi_2(s,\varepsilon)$ is uniform for $\varepsilon \neq 0$. Observe that \begin{align*}
 \frac{d\sigma_{W(h_{K_{\varepsilon}}+sw_{\varepsilon})} - d\sigma_{K_{\varepsilon}}}{s} \to d{\sigma}_{\varepsilon},
\end{align*} where $d{\sigma}_{\varepsilon}$ is some second surface area measure on the sphere. 

Therefore, \begin{align*}
    \lim_{s\to 0}\Pi_2(s,\varepsilon) &= \int_{S^{n-1}}w_{\varepsilon}(\theta) g(\nu_{K_{\varepsilon}}^{-1}(\theta)) d\sigma_{\varepsilon}(\theta).
\end{align*} To see that the convergence is uniform for $\varepsilon \neq 0$, we write \begin{align*}
    &\left|\frac{1}{s}\int_{S^{n-1}}w_{\varepsilon}(\theta)g(\nu_{K_{\varepsilon}}^{-1}(\theta))\left(d\sigma_{W(h_{K_{\varepsilon}}+sw_{\varepsilon})} - d\sigma_{K_{\varepsilon}} \right)(\theta) - \int_{S^{n-1}} w_{\varepsilon}(\theta) g(\nu_{K_{\varepsilon}}^{-1}(\theta)) d\sigma_{\varepsilon}(\theta) \right| \\ &\le \norm{w_{\varepsilon}}_{\infty} \norm{g}_{\infty} \left|\frac{d\sigma_{W(h_{K_{\varepsilon}} + sw_{\varepsilon})} - d\sigma_{K_{\varepsilon}} }{s} - d\sigma_{\varepsilon} \right|(S^{n-1}) \\ &\lesssim \left|\frac{d\sigma_{W(h_{K_{\varepsilon}} + sw_{\varepsilon})} - d\sigma_{K_{\varepsilon}} }{s} - d\sigma_{\varepsilon} \right|(S^{n-1}) \\ &\lesssim O_{K,L}(s),
\end{align*} where the last step is simply a consequence of the Steiner-type formula for surface area measures. See Section 4.2 in Schneider \cite{book4}.

Therefore, indeed \begin{align*}
&\lim_{\varepsilon \to 0} \frac{d^2}{ds^2} \mu(W(h_{K_{\varepsilon}} + sw_{\varepsilon}))|_{s= 0} \\ &=
    \lim_{\varepsilon \to 0} \lim_{s\to 0} \frac{1}{s}\left( \frac{d}{ds}\mu(W(h_{K_{\varepsilon}} + sw_{\varepsilon})\big|_{s} - \frac{d}{ds}\mu(W(h_{K_{\varepsilon}}+sw_{\varepsilon}))|_{s=0}\right) \\ &= \lim_{s\to 0}\lim_{\varepsilon \to 0} \frac{1}{s}\left( \frac{d}{ds}\mu(W(h_{K_{\varepsilon}} + sw_{\varepsilon})\big|_{s} - \frac{d}{ds}\mu(W(h_{K_{\varepsilon}}+sw_{\varepsilon}))|_{s=0}\right) \\ &= \lim_{s\to 0}\lim_{\varepsilon \to 0} \frac{1}{s}\left(\int_{S^{n-1}} w_{\varepsilon}(\theta) d\sigma_{\mu, W(h_{K_{\varepsilon}} + s w_{\varepsilon})}(\theta) -\int_{S^{n-1}}w_{\varepsilon}(\theta)\sigma_{\mu, K_{\varepsilon}}(\theta) \right)\\ &= \lim_{s\to 0} \frac{1}{s}\left(\int_{S^{n-1}} w(\theta) d\sigma_{\mu, W(h_K + sw)}(\theta) - \int_{S^{n-1}}w(\theta) d\sigma_{\mu, K}(\theta) \right).
\end{align*} Since $((1-s)h_K +_p sh_L) = h_K + sw + O(s^2)$, it follows from Lemma \ref{hausdorffsurfacearea} that \begin{align*}
    \int_{S^{n-1}} w(\theta) d\sigma_{\mu, W(h_K+sw)}(\theta) &= \int_{S^{n-1}} w(\theta) d\sigma_{\mu, W((1-s)h_K +_p sh_L)}(\theta) + O(s^2).
\end{align*} Therefore, \begin{align*}
    \lim_{\varepsilon \to 0}\frac{d^2}{ds^2} \mu(W(h_{K_{\varepsilon} + sw_{\varepsilon}}))|_{s=0} &= \lim_{s\to 0}\frac{1}{s}\left(\int_{S^{n-1}}w(\theta) d\sigma_{\mu, K_s}(\theta) - \int_{S^{n-1}}w(\theta) d\sigma_{\mu, K}(\theta) + O(s^2) \right) \\ &= \frac{d}{ds} \int_{S^{n-1}}w(\theta) d\sigma_{\mu, K_s}(\theta)\big|_{s=0} \\ &= \frac{d}{ds} \sum_{i=1}^{N} w(u_i)\mu_{n-1}(F_i(K_s))\big|_{s=0} \\ &= \sum_{i=1}^{N}h_i s_i \frac{d}{ds} \mu_{n-1}(F_i(K_{s}))\big|_{s=0}.
\end{align*}This concludes the proof of (\ref{secderivlimit}) and the lemma.
\end{proof}

We will use as an important ingredient the following fact proven in Proposition 3.6 of Putterman \cite{Put}.

\begin{lemma}[Putterman \cite{Put} ]\label{intervals}
Let $K, L$ be strongly isomorphic polytopes. There exist finitely many open intervals $I_1, ..., I_m \subset [0,1]$ such that $[0,1]\setminus \cup_{j=1}^{m}I_j$ is a finite set of points, and for each $j$, all the polytopes $K_{\lambda}$ for $\lambda \in I_j$ are strongly isomorphic.
\end{lemma}

Moreover, as follows from Putterman's proof, at points $p$ not contained in $\cup_{j=1}^{m} I_j$, we have that the face in $K_p$ corresponding to some normal vector now vanishes, while a face corresponding to this vector was present in the polytopes in at least one of the intervals adjacent to $p$.

This allows us to demonstrate the following:

\begin{proposition}
Assume that (\ref{local-key-eq}) holds for some $p, q < 1$ for any $C^{2,+}$ symmetric convex $K \in \mathcal{F}$ and any even $C^1$-smooth $f:\partial K\rightarrow \R$. Then for any two strongly isomorphic symmetric polytopes $K, L \in \mathcal{F}$, $\mu(K_{\lambda})^{\frac{q}{n}}$ is concave on $[0,1]$.
\end{proposition}
\begin{proof} We apply Lemma \ref{intervals} to get a sequence of intervals $I_1, ..., I_m$. From the remark, for $p \not\in \cup_{j=1}^{m}I_j$, a face corresponding to some normal just vanishes. However, even if this face happens to be a facet (an $(n-1)-$dimensional face), our computation in (\ref{firstder}) is unaffected. We simply have that $\mu_{n-1}(F_i(K_{\lambda})) = 0$ for some of the $i$. Therefore, considering $\lambda\to p$ from the left and from the right separately, the formula in (\ref{firstder}) shows that $\frac{d}{d\lambda} \mu(K_{\lambda})$ is continuous at all points $p \in [0,1]\setminus \cup_{j=1}^{m}I_j$. Since formula (\ref{firstder}) implies continuity in $\cup_{j=1}^{m}I_j$ also, we see that $\frac{d}{d\lambda}\mu(K_{\lambda})$ is continuous on the whole interval $[0,1]$.

By Lemma 3.5, $\frac{d}{d\lambda}\left(\mu(K_{\lambda})^{\frac{q}{n}}\right)$ is nonincreasing on the intervals $I_i$, $i = 1, ..., m$. Since $\frac{d}{d\lambda} \left(\mu(K_{\lambda})^{\frac{q}{n}}\right)$ is continuous, it must therefore be nonincreasing on the whole interval $[0,1]$. In other words, $\mu(K_{\lambda})^{\frac{q}{n}}$ is concave on $[0,1]$, as desired.
\end{proof}

\begin{proof}[Proof of Theorem 3.1] Since any two symmetric convex bodies $K, L$ can be approximated by sequences of strongly isomorphic polytopes converging in the Hausdorff metric to $K, L$ respectively and $\mathcal{F}$ is open with respect to this metric, and moreover the pointwise limit of concave functions is concave, we deduce our theorem from Proposition 3.7.
\end{proof}

\section{Proof of Proposition \ref{implications}.} 

Recall that a measure is said to be ray-decreasing, if its density $f$ satisfies $f(tv)\geq f(v)$, for any $v\in\R^n$ and any $t\in[0,1].$ In particular, a density of any even log-concave measure is ray-decreasing. Further, let us assume without loss of generality that the density of the measure is $C^2-$smooth.

In view of Lemma \ref{local}, the local version of the $(p,p)-$ conjecture reads as
$$
\int_{{\partial K}} H_x f^2 d \mu - \int_{{\partial K}} \langle II^{-1} \nabla_{{\partial K}} f, \nabla_{{\partial K}} f \rangle d \mu
+ (1-p) \int_{{\partial K}} \frac{f^2}{h_K(n_x)} d\mu
 \le \frac{1}{\mu(K)} \Bigl( 1- \frac{p}{n}\Bigr) \Bigl( \int_{\partial K} f d \mu \Bigr)^2.
$$
By Theorem \ref{loc->glob}, the local version is equivalent to the global version. Therefore, in order to show that the $(p,p)$-inequality strengthens when $p$ decreases, it is enough to show that
\begin{equation}\label{toshow}
\mu(K)\int_{\partial K}\frac{f^2}{\langle x,n_x\rangle}d\mu\geq \frac{1}{n} \left(\int_{\partial K} fd\mu\right)^2,
\end{equation}
for any measure $\mu$ with a ray-decreasing smooth density. 

We shall verify (\ref{toshow}). We write (see, e.g., Nazarov \cite{fedia} as well as Livshyts \cite{Liv}, \cite{Liv1}, \cite{Liv2}):
$$\mu(K)=\int_0^1\int_{\partial K} \langle x,n_x\rangle t^{n-1} e^{-V(tx)}dt dH_{n-1}(x).$$
As the density $e^{-V}$ is ray-decreasing, the function $V$ is ray-increasing, and thus we see that $V(tx)\leq V(x)$, for all $t\in [0,1]$. We conclude that
$$\mu(K)\geq \frac{1}{n} \int_{\partial K} \langle x,n_x\rangle d\mu,$$
which, together with Cauchy's inequality, implies (\ref{toshow}). $\square$

\section{An application of Bochner's method and integration by parts}

Consider an even measure $\mu$ on $\R^n$ with $C^2$ density $d\mu(x)=e^{-V(x)}dx$, and fix a $C^{2,+}-$smooth symmetric convex set $K$. In this section and everywhere below, we use notation
$$\int:=\frac{1}{\mu(K)}\int_K d\mu(x).$$
We shall also use the notation
$$Var(g)=\int g^2-\left(\int g\right)^2.$$
Let 
$$Lu=\Delta u-\langle \nabla u, \nabla V\rangle.$$ 

The following Bochner-type identity was obtained by Kolesnikov and Milman \cite{KM1}. It is a particular case of Theorem 1.1 in \cite{KM1} (note that $\rm{Ric}_{\mu}  = \nabla^2 \it{V}$ in our case). This is a generalization of a classical result of R.C.~Reilly.

\begin{proposition}[Kolesnikov-Milman \cite{KM1}]\label{raileyprop}
Let $u \in C^2(K)$  and $u_{n} = \langle \nabla u, n_x \rangle \in C^1(\partial K)$. Then  
\begin{align}
\label{railey}
\int_{K} (L u)^2 d \mu & =\int_{K} \left(||\nabla^2 u||^2+\langle\nabla^2 V \nabla u, \nabla u\rangle\right) d \mu+
\\&  \nonumber\int_{\partial K}  (H_x u_n^2 -2\langle \nabla_{\partial K} u, \nabla_{\partial K}  u_n\rangle +\langle \mbox{\rm{II}} \nabla_{\partial K} u, \nabla_{\partial K}  u\rangle )  \,d\mu_{\partial K} (x).
\end{align}
\end{proposition}

Therefore, we get:

\begin{lemma}\label{pqconj-byparts}
Suppose for every even $f\in C^2(\partial K)$ there exists $u\in C^2(K)$ such that for each $x\in\partial K,$
$$\langle \nabla u,n_x\rangle=f(x),$$
and
$$\int ||\nabla^2 u||^2+\langle\nabla^2 V \nabla u, \nabla u\rangle \geq Var (Lu)+\frac{q}{n}\left(\int Lu\right)^2+\frac{1-p}{\mu(K)}\int_{\partial K} \frac{\langle \nabla u, n_x\rangle^2}{\langle x, n_x\rangle}d\mu_{\partial K}.$$
Then for every $C^2$-smooth symmetric convex set $L$, and every $\lambda\in[0,1]$, 
$$\mu(\lambda K+_p(1-\lambda)L)^{\frac{q}{n}}\geq \lambda \mu(K)^{\frac{q}{n}}+(1-\lambda)\mu(L)^{\frac{q}{n}}.$$
\end{lemma}
\begin{proof} Recall that for any positive definite $n\times n$ matrix $A$ and for any $x,y\in\R^n$ we have
\begin{equation}\label{matrix}
\langle Ax,x\rangle+\langle A^{-1}y,y\rangle\geq 2 \langle x,y\rangle.
\end{equation}
As $K$ is convex, its second quadratic form $\rm{II}$ is positive definite, and consequently, 
\begin{equation}\label{sqf}
-2\langle \nabla_{\partial K} u, \nabla_{\partial K}  u_n\rangle +\langle \mbox{\rm{II}} \nabla_{\partial K} u, \nabla_{\partial K}  u\rangle\geq -\langle \mbox{\rm{II}}^{-1}\nabla_{\partial K}  f,  \nabla_{\partial K}  f\rangle.
\end{equation}
Recall also that 
\begin{equation}\label{diverg}
\int_K Lu d\mu= \int_{\partial K} \langle \nabla u,n_x\rangle d\mu_{\partial K}.
\end{equation}
By (\ref{sqf}), (\ref{diverg}) and Proposition \ref{raileyprop}, the assumption of this Lemma implies the validity of the local version of Conjecture \ref{mainconj}, as per Lemma \ref{local}. The Lemma thus follows from Theorem \ref{loc->glob}. 
\end{proof}

\section{Preparatory estimates}

Fix a measure $\mu$ with even density $e^{-V}$ as in Theorem \ref{main}. Suppose \begin{align}rB_2^n\subset K\subset R B_2^n.
\end{align}
Let $C_{poin}(K;\mu)$ be the Poincare constant of the restriction of $\mu$ on $K,$ that is the smallest non-negative number such that
$$\int_K f^2 d\mu-\frac{1}{\mu(K)}\left(\int_K f d\mu\right)^2\leq C^2_{poin}(K;\mu) \int_{K}|\nabla f|^2d\mu,$$
for every differentiable function $f:K\rightarrow\R.$ We recall that $C^{-2}_{poin}(K;\mu)$ is the first Neumann eigenvalue of the operator $Lu=\Delta u-\langle \nabla u,\nabla V\rangle$ restricted on $K,$ or in other words, $C_{poin}(\mu;K)$ is the smallest number such that there exists a function $u$ with $Lu=-C_{poin}(\mu;K)^{-2} u$ on $K$ and $\langle \nabla u,n_x\rangle=0$ on $\partial K.$ The celebrated Kannan-Lovasz-Simonovits (KLS) conjecture \cite{KLS} states that the Poincare constant of an isotropic log-concave measure is bounded from above by an absolute constant, independent of the dimension (in our current notation, isotropicity means that the restriction of $\mu$ onto $K$ is isotropic). See Lee-Vempala \cite{LV} for the best to date estimate in the direction of this conjecture, and for the discussion of history and powerful implications of the KLS conjecture.

The following fact is classical and appears, e.g. in Lemma 5.1 from \cite{KolLiv}.

\begin{lemma}\label{nabla V}
For any symmetric convex set $K$, and even log-concave measure $\mu$ with even density $e^{-V}$,
$$\frac{1}{\mu(K)}\int_K |\nabla V|^2d\mu\leq \int \Delta Vd\mu.$$
\end{lemma}

Next, we show, using ideas similar to \cite{KolLiv}:

\begin{lemma}\label{estimate1-upd}
Suppose $K$ is a set. Let $u:K\rightarrow \R$ be a $C^2$-smooth function, and fix $a,b>0.$ Then
$$ a||\nabla^2 u||^2+b|\nabla u|^2 \geq \frac{ab(Lu)^2}{a|\nabla V|^2+bn}.$$
\end{lemma}
\begin{proof} Assume without loss of generality that $a=1$. By the Cauchy-Schwarz inequality,
\begin{equation}\label{hessian}
 ||\nabla^2 u||^2 \geq \frac{1}{n}  |\Delta u|^2 .
\end{equation}
Indeed, recall that $||\nabla^2 u||^2=\sum_{i=1}^n \lambda^2_i$, where $\lambda_1,...,\lambda_n$ are the eigenvalues of $\nabla^2 u$, and recall also that $\Delta u=\sum_{i=1}^n \lambda_i$. Hence (\ref{hessian}) follows.

Next, writing $\Delta u=Lu+\langle \nabla V,\nabla u\rangle$, we see that
$$||\nabla^2 u||^2+b|\nabla u|^2 \geq \frac{1}{n}\left(Lu+\langle \nabla V,\nabla u\rangle\right)^2+b|\nabla u|^2=$$
\begin{equation}\label{estab}
\langle A \nabla u,\nabla u\rangle+2\langle \frac{Lu}{n}\nabla V, \nabla u\rangle+\frac{(Lu)^2}{n},
\end{equation}
where
$$A=\frac{1}{n}\nabla V\otimes \nabla V+b Id.$$
We observe that for any vector $z\in\R^n$ and for all $\alpha, \beta\in \R,$
$$
\left( \beta Id+\alpha z\otimes z\right)^{-1}z=\frac{z}{\beta+\alpha|z|^2}.
$$
Using this observation, and the inequality (\ref{matrix}) with $A$ defined above (as it is indeed positive definite), $x=-\nabla u$ and $y=\frac{Lu}{n}\nabla V$, we estimate (\ref{estab}) from below by
$$\frac{(Lu)^2}{n}\left(1 - \frac{1}{n}\langle A^{-1} \nabla V, \nabla V \rangle\right) = \frac{b(Lu)^2}{bn+|\nabla V|^2}.$$

Rescaling finishes the proof.
\end{proof}

From Lemmas \ref{nabla V} and \ref{estimate1-upd} we get

\begin{lemma}\label{estimate3-upd}
Suppose $K$ is a symmetric convex set. Let $u:K\rightarrow \R$ be an even $C^2$-smooth function such that $Lu=1$ on $K$, and fix $a, b\in\R$ such that $a C^{-2}_{poin}(K,\mu)+b\geq 0$ and $a\geq b.$ Then
$$ \int a||\nabla^2 u||^2+b|\nabla u|^2 \geq \frac{C^{-2}_{poin}(K,\mu)a+b}{\left(1+C^{-2}_{poin}(K,\mu)\right)(1+k_2)n}.$$
\end{lemma}
\begin{proof} Since $u$ is even, 
$$\int \nabla u=0,$$
and we apply the Poincare inequality:
$$
\int ||\nabla^2 u||^2\geq C^{-2}_{poin}(K,\mu)\int |\nabla u|^2.
$$
We estimate
\begin{equation}\label{poin-sect5}
\int a||\nabla^2 u||^2+b|\nabla u|^2 \geq  \int (a-\epsilon)||\nabla^2 u||^2+(b+\epsilon C^{-2}_{poin}(K,\mu))|\nabla u|^2.
\end{equation}
Pick $\epsilon=\frac{a-b}{1+ C^{-2}_{poin}(K,\mu)}$. Note that $\epsilon\geq 0$ since $a\geq b.$ We get 
$$a-\epsilon=b+\epsilon C^{-2}_{poin}(K,\mu)=\frac{b+aC^{-2}_{poin}(K,\mu)}{1+C^{-2}_{poin}(K,\mu)}.$$ 
We combine (\ref{poin-sect5}) with Lemma \ref{estimate1-upd} (applied with the parameters $a-\epsilon$ and $b+\epsilon C^{-2}_{poin}(K,\mu)$), to get
$$ \int a||\nabla^2 u||^2+b|\nabla u|^2 \geq \int\frac{C^{-2}_{poin}(K,\mu)a+b}{\left(1+C^{-2}_{poin}(K,\mu)\right)(n+|\nabla V|^2)}.$$
Next, we use Jensen's inequality and recall that \begin{align}\label{gradvineq} \int |\nabla V|^2\leq k_2n\end{align} by Lemma \ref{nabla V} (in view of the definition of $k_2$). The Lemma follows.
\end{proof}

\section{Proof of Theorem \ref{lebesgue}.}

Let $K$ be a symmetric convex set, and denote by $C_{poin}(K)$ the Poincare constant of the restriction of the Lebesgue measure on $K.$



\begin{proposition}\label{propnew}
For every convex symmetric $C^2$-smooth set $K$, and for every non-negative even $C^1$-smooth function $f:\partial K\rightarrow\R$, there exists a $C^2$-smooth function $u: K\rightarrow\R$ such that $\langle \nabla u,n_x\rangle=f(x)$ for all $x\in\partial K,$ and such that
$$\int ||\nabla^2 u||^2\geq Var(\Delta u)+\frac{1-p}{|K|}\int_{\partial K}\frac{\langle \nabla u,n_x\rangle^2}{\langle x,n_x\rangle},$$
whenever
$$p\geq \max\left(1-\frac{r}{C_{poin}(K)(\sqrt{n}+1)},0\right).$$
\end{proposition}
\begin{proof} We may assume that $f$ is not identically zero, and thus by continuity, $\int_{\partial K} f>0.$ Without loss of generality, by scaling, we may assume that 
\begin{equation}\label{wlog}
\int_{\partial K} f=|K|.
\end{equation}

Let $u:K\rightarrow\R$ be such a function that
$$\langle \nabla u,n_x\rangle=f(x)\,\,\,\,\forall\,x\in\partial K$$
and
$$\Delta u=1.$$
By (\ref{wlog}) this system is compatible, and it has a solution. Further, by the standard regularity results (see, e.g. Evans \cite{evans}), this solution is twice differentiable. Moreover, since $K$, $F$ and $f$ are even, the solution is even as well \cite{evans}.

We estimate, using $|\langle \nabla u, n_x\rangle| \le |\nabla u|$,
\begin{equation}\label{John}
\int_{\partial K} \frac{\langle \nabla u, n_x\rangle^2}{\langle x,n_x\rangle}\leq \frac{1}{r}\int_{\partial K} {\langle |\nabla u| \nabla u, n_x\rangle}=\frac{1}{r}\int_K div(|\nabla u|\nabla u),
\end{equation}
where in the last line we used the divergence theorem, and we also used the fact that $\langle \nabla u, n_x\rangle=f\geq 0$, and hence $|\langle\nabla u, n_x\rangle|=\langle\nabla u, n_x\rangle$. We write
$$div(|\nabla u|\nabla u)=|\nabla u|\Delta u+\frac{1}{|\nabla u|}\langle \nabla^2 u{\nabla} u,\nabla u\rangle=|\nabla u|+\frac{1}{|\nabla u|}\langle \nabla^2 u{\nabla} u,\nabla u\rangle,$$
where we used that $\Delta u=1$. Next, we estimate
$$
\frac{1-p}{|K|}\int_{\partial K} \frac{\langle \nabla u, n_x\rangle^2}{\langle x,n_x\rangle}\leq \frac{1-p}{r}\int |\nabla u|+\frac{1}{2}\left(\frac{1}{\alpha} ||\nabla^2 u||^2+\alpha |\nabla u|^2\right),
$$
for any $\alpha>0$. By Cauchy's inequality, the above is bounded by
\begin{equation}\label{bnd2-1}
\frac{1-p}{r}\sqrt{\int |\nabla u|^2}+\frac{1-p}{2r}\int\left(\frac{1}{\alpha} ||\nabla^2 u||^2+\alpha |\nabla u|^2\right).
\end{equation}
Lastly, since $u$ is an even function,
$$\int |\nabla u|^2\leq C^2_{poin} \int ||\nabla^2 u||^2,$$
where $C_{poin}=C_{poin}(K)$ is the Poincare constant of $K$. Therefore, selecting $\alpha=C^{-1}_{poin}$, we get that (\ref{bnd2-1}) is bounded by
\begin{equation}\label{bnd3-1}
\frac{(1-p)C_{poin}}{r}\sqrt{\int ||\nabla^2 u||^2}+\frac{(1-p)C_{poin}}{r}\int ||\nabla^2 u||^2.
\end{equation}
Since $\Delta u=1$, we have $Var(\Delta u)=0$, and using (\ref{bnd2-1}) we see, that our goal is
\begin{equation}\label{goal-LEB}
\sqrt{\int ||\nabla^2 u||^2}+\int ||\nabla^2 u||^2\leq \frac{r}{(1-p)C_{poin}}\int ||\nabla^2 u||^2.
\end{equation}
As $||\nabla^2 u||^2\geq \frac{1}{n}(\Delta u)^2=\frac{1}{n}$, we see that (\ref{goal-LEB}) is indeed correct whenever
$$p\geq \max\left(1-\frac{r}{C_{poin}(K)(\sqrt{n}+1)},0\right). \text{   } $$

\end{proof}

\textbf{Proof of Theorem \ref{lebesgue}. }
Consider $K$ a symmetric $C^2-$smooth convex body in $\mathbb{R}^n$. Choose $T \in GL_n$ such that $TK$ is in isotropic position. For isotropic convex bodies, $r\geq 1+o(1)$, as shown by Kannan, Lovasz, and Simonovits \cite{KLS}, and $C_{poin}(TK) \leq cn^{\frac{1}{4}}$, as shown by Lee, Vempala \cite{LV}. Thus the conclusion of Proposition \ref{propnew} holds for $TK$ with $p \geq 1 - Cn^{-0.75}$. From Proposition \ref{raileyprop}, we can write the conclusion of Proposition 7.1 as \begin{align*}
    \int_{\partial TK}H_xu_n^2 -2\langle \nabla_{\partial TK} u, \nabla_{\partial TK}u_n\rangle + \langle \mbox{\rm{II}}\nabla_{\partial TK}u, \nabla_{\partial TK}u\rangle + (1-p)\frac{u_n^2}{\langle x,n_x\rangle}dx &\le \int_{TK} (Lu)^2 dx.
\end{align*} 
Here, the last term is simply $|TK|.$

Therefore, since the quadratic form $\rm{II}$ is positive definite, for any nonnegative even $C^1-$smooth $f: \partial TK \to \mathbb{R}$, we have $$\int_{\partial TK} H_x f^2 - \langle \mbox{\rm{II}}^{-1}\nabla_{\partial TK} f, \nabla_{\partial TK} f\rangle + (1-p)\frac{f^2}{\langle x, n_{x}\rangle} dx - \frac{1}{|TK|}\left(\int_{\partial TK}f dx \right)^2 \le 0.$$ By the argument of Lemma \ref{local}, this statement is equivalent to \begin{align*}
    \frac{d^2}{d\varepsilon^2}\log |TK+_p \varepsilon f| \le 0
\end{align*} for each nonnegative $C^1-$smooth $f: S^{n-1} \to \mathbb{R}$. Therefore, we also have \begin{align*}
    \frac{d^2}{d\varepsilon^2}\log |T^{-1}(TK+_p \varepsilon f)| \le 0
\end{align*} for each nonnegative even $C^1-$smooth $f: S^{n-1} \to \mathbb{R},$ in view of the fact that
$$|T^{-1}(TK+_p \varepsilon f)| = |\det T|^{-1}|TK+_p \varepsilon f|.$$ 
Following the arguments in Section 5 of Kolesnikov and Milman \cite{KolMil}, let us define $f_{T^{-1}}: S^{n-1} \to \mathbb{R}$ by $$f_{T^{-1}}(\theta) = f\left(\frac{(T^{-1})^t \theta}{|(T^{-1})^t \theta|} \right)|(T^{-1})^t \theta|.$$ Then, for small enough $\varepsilon>0$ and all $\theta \in S^{n-1}$, we have \begin{align*}
    h_{K +_p \varepsilon f_{T^{-1}}}(\theta) &= (h_K^p(\theta) + \varepsilon f_{T^{-1}}^p(\theta))^{\frac{1}{p}} \\ &= \left(h_{TK}^p\left(\frac{(T^{-1})^t \theta}{|(T^{-1})^t \theta|} \right) + \varepsilon f^p\left(\frac{(T^{-1})^t \theta}{|(T^{-1})^t \theta|} \right) \right)^{\frac{1}{p}}|(T^{-1})^t \theta| \\ &= (h_{TK}^p + \varepsilon f^p)^{\frac{1}{p}}\left(\frac{(T^{-1})^t \theta}{|(T^{-1})^t \theta|} \right)|(T^{-1})^t \theta| \\ &= h_{TK +_p \varepsilon f}((T^{-1})^t\theta) \\ &= h_{T^{-1}(TK +_p \varepsilon f)}(\theta),
\end{align*} 
where in the last passage we used the fact that for any linear map $A,$ any vector $y$ and any convex body $L,$
$$h_{AL}(y)=\sup_{z\in AL}\langle z,y\rangle=\sup_{x\in L}\langle x, (A^{-1})^ty\rangle=h_L((A^{-1})^ty).$$
It follows that \begin{align*}
    \frac{d^2}{d\varepsilon^2}\log |K +_p \varepsilon f_{T^{-1}}| &\le 0
\end{align*} for each nonnegative even $C^1-$smooth $f: S^{n-1} \to \mathbb{R}$. Since $f \to f_{T^{-1}}$ is a bijection on the set of nonnegative even $C^{-1}-$smooth functions on $S^{n-1}$, we have that \begin{align*}
    \frac{d^2}{d\varepsilon^2}\log |K +_p \varepsilon f| &\le 0
\end{align*} for each nonnegative even $C^1-$smooth $f: S^{n-1} \to \mathbb{R}$.


To finish, we may apply the procedure of Theorem 3.1. While our local inequality only holds for $f$ nonnegative, this is sufficient to conclude the global inequality for $K \subset L$. To see this, take our approximations $K_{\varepsilon}, L_{\varepsilon}$ in Lemma 3.6 such that $K_{\varepsilon} \subset K$ and $L \subset L_{\varepsilon}$ and recall that our choice of $f$ in the local inequality is \begin{align*}
    f_{\varepsilon}(x) &= \frac{1}{p}h_{K_{\varepsilon}}(n_x)\left(\left(\frac{h_{L_{\varepsilon}}(n_x)}{h_{K_{\varepsilon}}(n_x)} \right)^p - 1 \right),
\end{align*} which is non-negative. It remains to recall that $p-$Minkowski interpolations preserve inclusions.
 $\square$

\begin{remark}\label{explain}
We note that Kolesnikov and Milman \cite{KolMilsupernew} used the estimate 
$$\int_{\partial K} \frac{\langle \nabla u, n_x\rangle^2}{\langle x,n_x\rangle}\leq \frac{1}{r^2}\int_{\partial K} {\langle |\nabla u|^2 x, n_x\rangle}=\frac{1}{r^2}\int_K div(|\nabla u|^2x),$$
in place of (\ref{John}), which is rougher, and hence leads to the rougher bound $p\geq 1-cn^{-1.5}$. However, (\ref{John}) only works for non-negative functions, hence our result is only valid in the partial case $K\subset L.$ We note also that the form of the inequality which we prove is not invariant under the transformation $L\rightarrow tL,$ unlike the additive version of the conjecture, and hence we cannot assume that $K\subset L$ without loss of generality.  
\end{remark}

\section{Proof of Theorem \ref{main}.}

For brevity, we will sometimes write $C_{poin}=C_{poin}(K,\mu),$ for the Poincare constant of the restriction of $\mu$ on $K$ (which was defined in Section 6.)

\begin{proposition}\label{keyprop-1}
Let $K$ be a convex set in $\R^n$ containing $rB_2^n$. Then for every $f\in C^1(\partial K)$ there exists $u\in C^2(K)$ such that for each $x\in\partial K,$
$$\langle \nabla u,n_x\rangle=f(x),$$
and
\begin{equation}\label{goal-prop1}
\int ||\nabla^2 u||^2+\langle \nabla^2V\nabla u,\nabla u\rangle \geq Var(Lu)+\frac{q}{n}\left(\int Lu\right)^2+
\end{equation}
$$
\frac{1-p}{\mu(K)}\int_{\partial K} \frac{\langle \nabla u, n_x\rangle^2}{\langle x, n_x\rangle}d\mu_{\partial K},
$$
provided that at least one of the two conditions hold: 
\begin{itemize}
\item $k_1\in [\frac{1}{n},1]$ and
$$(1-p)\frac{1+n}{r^2}+q(1+k_2)(1+k_1)\leq 2k_1;$$
\item for all $k_1\geq 0,$
$$\begin{cases}
q\left(k_1^2+k_1k_2-(nk_1+k_2)\frac{1-p}{r^2}\right)\leq k_1^2+n\left(\frac{1-p}{r^2}\right)^2-\frac{1-p}{r^2}(n+1)k_1;\\
\frac{1-p}{r^2}\leq \frac{k_1}{n}.
\end{cases}$$
\end{itemize}
\end{proposition}
\begin{proof} Let $u$ be the solution of the Neumann system
$$\langle \nabla u,n_x\rangle=f(x),$$
and
$$Lu=\frac{\int_{\partial K}f d\mu_{\partial K}}{\mu(K)}.$$
Note that 
\begin{equation}\label{var-prop1}
Var(Lu)=0.
\end{equation}
Observing that $\langle x, n_x\rangle\geq r,$ $\langle \nabla u, n_x\rangle\leq |\nabla u|$ and using the divergence theorem (similarly to the argument in Remark \ref{explain}), we estimate
$$\int_{\partial K} \frac{\langle \nabla u, n_x\rangle^2}{\langle x, n_x\rangle}d\mu_{\partial K}\leq \frac{1}{r^2}\int_K div(|\nabla u|^2e^{-V}x)dx.$$
We write
\begin{equation}\label{div-prop1}
div(|\nabla u|^2e^{-V}x)=e^{-V}(|\nabla u|^2\left(n-\langle x,\nabla V\rangle\right)+2\langle \nabla^2 u\nabla u,x\rangle),
\end{equation}
and note that 
$$\langle x,\nabla V\rangle\geq k_1|x|^2.$$
Indeed, to see this, consider a function $g(t)=\langle x,\nabla V(tx)\rangle.$ By the intermediate value theorem, $g(1)-g(0)=g'(\xi)$, for some $\xi\in [0,1]$. Observe that $g'(\xi)=\langle \nabla^2 V(\xi x)x,x\rangle\geq k_1 |x|^2$, by our assumption that $\nabla^2 V\geq k_1 \rm{Id}.$ It remains to note that $g(0)=0$ since $V$ is a smooth convex even function.

Using the inequality $2ab\leq \frac{a^2}{t}+tb^2$, for all $t>0$, estimating the operator norm of $\nabla^2u$ with the Hilbert-Schmidt norm, and applying the Cauchy inequality, for any $\alpha>0$, we estimate (\ref{div-prop1}) by
$$
e^{-V}\left(|\nabla u|^2\left(n-(k_1-\alpha)|x|^2\right)+\frac{1}{\alpha}||\nabla^2 u||^2\right).
$$
Therefore, (\ref{goal-prop1}) will follow in case we verify
$$\int ||\nabla^2 u||^2+k_1|\nabla u|^2 \geq \frac{q}{n}+\frac{1-p}{r^2}\int |\nabla u|^2\left(n-(k_1-\alpha)|x|^2\right)+\frac{1}{\alpha}||\nabla^2 u||^2.
$$
Denote $\theta=\frac{1-p}{r^2}$. We let $\alpha=k_1$, and the inequality becomes
$$\int a||\nabla^2 u||^2+b|\nabla u|^2 \geq \frac{q}{n},$$
where 
$$a=1-\frac{\theta}{k_1},$$
$$b=k_1-\theta n.$$
\textbf{Case 1.} Suppose $k_1\in [\frac{1}{n},1].$ In this case, $a\geq b,$ and we are in a position to employ Lemma \ref{estimate3-upd}, provided that we also verify the condition $C^{-2}_{poin}(K,\mu)a+b\geq 0.$ The restriction on $q$ and $\theta$ then reads
$$\frac{C^{-2}_{poin}(K,\mu)a+b}{\left(1+C^{-2}_{poin}(K,\mu)\right)(1+k_2)n}\geq \frac{q}{n},$$
Recall that the Brascamp-Lieb inequality yields that for any convex set $K,$ we have $C^{-2}_{poin}(K,\mu)\geq k_1$ (see \cite{BrLi}.) Therefore, the inequality amounts to
\begin{equation}\label{ass3}
(1-p)\frac{1+n}{r^2}+q(1+k_2)(1+k_1)\leq 2k_1.
\end{equation}
\textbf{Case 2.} Suppose $k_1\in [0,\frac{1}{n}).$ The Lemma \ref{estimate3-upd} is not applicable, and therefore we employ Lemma \ref{estimate1-upd}, which yields, together with Jensen's inequality, that 
$$\int a||\nabla^2 u||^2+b|\nabla u|^2 \geq \frac{1}{n}\cdot\frac{ab}{ak_2+b},$$
provided that $a\geq 0$ and $b\geq 0.$ With our choice of parameters, the latter assumption boils down to
\begin{equation}\label{assump-case2}
\frac{1-p}{r^2}\leq \frac{k_1}{n},
\end{equation}
which implies that $a\geq 0$, and the restriction on $p$ and $q$ becomes
$$q\leq \frac{ab}{ak_2+b}=\frac{(1-\frac{\theta}{k_1})(k_1-\theta n)}{(1-\frac{\theta}{k_1})k_2+k_1-\theta n},$$
or equivalently, since the denominator is non-negative in view of (\ref{assump-case2}),
$$q\left(k_1^2+k_1k_2-(nk_1+k_2)\frac{1-p}{r^2}\right)\leq k_1^2+n\left(\frac{1-p}{r^2}\right)^2-\frac{1-p}{r^2}(n+1)k_1.$$
\end{proof}

Theorem \ref{main} follows from Proposition \ref{keyprop-1} and Lemma \ref{pqconj-byparts}.

\begin{remark} More generally, we get the result under the assumptions:
$$(1-p)\frac{1/k_1-n}{r^2}\leq 1-k_1$$
and
$$(1-p)\frac{C^{-2}_{poin}/k_1+n}{r^2}+q(1+k_2)(1+C^{-2}_{poin})\leq k_1+C^{-2}_{poin}$$
Alternatively, in case $K$ and $L$ are additionally contained in $R B_2^n$, and assuming that $R\leq \frac{C^{-1}_{poin}(K,\mu)}{k_1}$, we get the result under the assumptions
$$(1-p)\frac{2RC^{-1}_{poin}+n-k_1 R^2}{r^2}+q(1+k_2)(1+C^{-2}_{poin})\leq k_1+C^{-2}_{poin}.$$
and
$$(1-p)\frac{k_1 R^2-n}{r^2}\leq 1-k_1.$$
We skip the computation for the sake of brevity.
\end{remark}

\section{Proof of Proposition \ref{prop-main}.}

\begin{proposition}\label{keyprop-2}
Let $K$ be a symmetric convex set in $\R^n$ containing $rB_2^n$. Then for every non-negative $f\in C^1(\partial K)$ there exists $u\in C^2(K)$ such that for each $x\in\partial K,$
$$\langle \nabla u,n_x\rangle=f(x),$$
and
\begin{equation}\label{goal}
\int ||\nabla^2 u||^2+\langle \nabla^2V\nabla u,\nabla u\rangle \geq \int (Lu)^2-\left(\int Lu\right)^2+\frac{q}{n}\left(\int Lu\right)^2+
\end{equation}
$$\frac{1-p}{\mu(K)}\int_{\partial K} \frac{\langle \nabla u, n_x\rangle^2}{\langle x, n_x\rangle}d\mu_{\partial K},
$$
whenever 
$$(1-p)\frac{2\sqrt{n}\sqrt{1+k_2}\sqrt{1+k_1}+\sqrt{k_1}}{2r}+q(1+k_2)(1+k_1)\leq 2k_1$$
and
$$k_1\leq 1.$$
\end{proposition}
\begin{proof} Let $u$ be the solution of the Neumann system
$$\langle \nabla u,n_x\rangle=f(x),$$
and
$$Lu=\frac{\int_{\partial K}f d\mu_{\partial K}}{\mu(K)}.$$
Note that 
\begin{equation}\label{var}
Var(Lu)=0.
\end{equation}

Observing that $\langle x, n_x\rangle\geq r,$ $\langle \nabla u, n_x\rangle\leq |\nabla u|$ and using the divergence theorem, we estimate
$$\int_{\partial K} \frac{\langle \nabla u, n_x\rangle^2}{\langle x, n_x\rangle}d\mu_{\partial K}\leq \frac{1}{r}\int_K div(|\nabla u|e^{-V}\nabla u)dx.$$
We observe that
\begin{equation}\label{div}
div(|\nabla u|e^{-V}\nabla u)=e^{-V}\left(|\nabla u|Lu+\frac{1}{|\nabla u|}\langle \nabla^2 u\nabla u,\nabla u\rangle\right).
\end{equation}
Using the inequalities $2ab\leq \frac{a^2}{t}+tb^2$, for all $t>0$, estimating the operator norm of $\nabla^2u$ with the Hilbert-Schmidt norm, and applying the Cauchy's inequality, for any $\alpha,\beta>0$, we estimate (\ref{div}) by
$$
\frac{1}{2}e^{-V}\left(|\nabla u|^2(\alpha+\beta)+\frac{(Lu)^2}{\alpha}+\frac{||\nabla^2 u||^2}{\beta}\right).
$$
Without loss of generality, since (\ref{goal}) is scale-invariant, we may assume that 
$$Lu=\frac{\int_{\partial K}f d\mu_{\partial K}}{\mu(K)}=1.$$
Since $\nabla^2 V\geq k_1 Id$, we have $\langle \nabla^2 V\nabla u,\nabla u\rangle\geq k_1 |\nabla u|^2$. Therefore, in view of (\ref{var}), the inequality (\ref{goal}) will follow from
\begin{equation}\label{goal2}
\int ||\nabla^2 u||^2+k_1|\nabla u|^2 \geq \frac{q}{n}+ \frac{1-p}{2r}\int |\nabla u|^2(\alpha+\beta)+\frac{1}{\alpha}+\frac{||\nabla^2 u||^2}{\beta}.
\end{equation}
In other words, we need to show 
\begin{equation}\label{goal3}
\int a||\nabla^2 u||^2+b|\nabla u|^2-c \geq 0,
\end{equation}
where, letting $\theta=\frac{1-p}{2r}$, we write 
$$a=1-\frac{\theta}{\beta},$$
$$b=k_1-\theta(\alpha+\beta),$$
$$c=\frac{\theta}{\alpha}+\frac{q}{n}.$$
It remains to apply Lemma \ref{estimate3-upd}, and the conditions on $\theta, q$ become:
$$
\frac{C^{-2}_{poin}(K,\mu)\left(1-\frac{\theta}{\beta}\right)+k_1-\theta(\alpha+\beta)}{\left(1+C^{-2}_{poin}(K,\mu)\right)(1+k_2)n}\geq \frac{\theta}{\alpha}+\frac{q}{n};$$
$$(1-\frac{\theta}{\beta})C^{-2}_{poin}(K)+k_1-\theta(\alpha+\beta)\geq 0;$$
$$1-\frac{\theta}{\beta}\geq k_1-\theta(\alpha+\beta).$$
Letting $\alpha=\sqrt{n}\sqrt{1+k_2}\sqrt{1+C^{-2}_{poin}}$ and $\beta=C^{-1}_{poin}$, we arrive at
$$(1-p)\frac{\sqrt{n}\sqrt{1+k_2}\sqrt{1+C^{-2}_{poin}}+C^{-1}_{poin}}{r}+q(1+k_2)(1+C^{-2}_{poin})\leq k_1+C^{-2}_{poin}$$
and
$$(1-p)\frac{C_{poin}-C^{-1}_{poin}-\sqrt{n}\sqrt{1+k_2}\sqrt{1+C^{-2}_{poin}}}{2r}\leq 1-k_1.$$
Recall that $C^{-2}_{poin}\geq k_1,$ and thus we could replace $C^{-1}_{poin}$ with $\sqrt{k_1}$ in the statement of Lemma \ref{estimate3-upd}, which would transform the above restrictions into 

$$(1-p)\frac{\sqrt{n}\sqrt{1+k_2}\sqrt{1+k_1}+\sqrt{k_1}}{r}+q(1+k_2)(1+k_1)\leq 2k_1$$
and
$$(1-p)\frac{\sqrt{k_1}-1/\sqrt{k_1}-\sqrt{n}\sqrt{1+k_2}\sqrt{1+k_1}}{2r}\leq 1-k_1.$$

The second of the required restrictions holds under the assumptions of the present Proposition, in view of the fact that $k_1\leq 1$, as in this case the left hand side is negative and the right hand side is non-negative. The first condition was assumed explicitly.
\end{proof}

Proposition \ref{prop-main} follows from Proposition \ref{keyprop-2} and Lemma \ref{pqconj-byparts}, in view of the fact that interpolations preserve inclusions.

\begin{remark}
More generally, we get the conclusion under the assumptions
$$(1-p)\frac{2\sqrt{n}\sqrt{1+k_2}\sqrt{1+C^{-2}_{poin}}+C^{-1}_{poin}}{2r}+q(1+k_2)(1+C^{-2}_{poin})\leq k_1+C^{-2}_{poin}$$
and
$$(1-p)\frac{C_{poin}-C^{-1}_{poin}-\sqrt{n}\sqrt{1+k_2}\sqrt{1+C^{-2}_{poin}}}{2r}\leq 1-k_1.$$
\end{remark}

\section{The $(p,p)$-Brunn-Minkowski inequality for dilates of symmetric convex sets in the case of the Gaussian measure.}

In this section we prove Theorem \ref{dilates}. Let $\gamma$ be the Gaussian measure, and fix $K$ to be an arbitrary convex set with the Gaussian barycenter at the origin. Denote
$$\int:=\frac{1}{\gamma(K)}\int_K d\gamma(x).$$

First, we recall
\begin{lemma}[Cordero-Erasquin, Fradelizi, Maurey \cite{CFM}]\label{bconj-anal}
$$\int |x|^4-\left(\int |x|^2\right)^2\leq 2\int |x|^2.$$
\end{lemma}

Lemma \ref{bconj-anal}, in conjunction with Lemma \ref{nabla V}, implies:

\begin{lemma}\label{cl2}
Pick any $p\in [0,1].$ Let
$$u(x)=\frac{|x|^2}{2}$$ 
on $K$. Let 
$$F=Lu=n-|x|^2$$ 
on $K.$ Then
$$
\int ||\nabla^2 u||^2+|\nabla u|^2\geq 
$$
\begin{equation}\label{maineq-end}
Var(F)+\frac{p}{n}\left(\int F\right)^2+\frac{1-p}{\gamma(K)}\int_{\partial K} \frac{\langle \nabla u,n_x\rangle^2}{\langle x,n_x\rangle}d\gamma_{\partial K}(x).
\end{equation}
\end{lemma}
\begin{proof} Integrating by parts, we see that
$$\int_{\partial K} \frac{\langle \nabla u,n_x\rangle^2}{\langle x,n_x\rangle}d\gamma_{\partial K}(x)=\gamma(K)\left(n-\int |x|^2\right).$$
In view of the fact that $Var(F)=Var(|x|^2),$ and the definition of $u,$ the inequality (\ref{maineq-end}) rewrites as
$$n+\int |x|^2\geq Var(|x|^2)+\frac{p}{n}\left(n-\int |x|^2\right)^2+(1-p)\left(n-\int |x|^2\right),$$
which, by Lemma \ref{bconj-anal}, follows from
$$0\leq p\int |x|^2-\frac{p}{n}\left(\int |x|^2\right)^2.$$
This, in turn, follows from Lemma \ref{nabla V}, applied with $V=\frac{|x|^2}{2}.$
\end{proof}

In view of Lemma \ref{pqconj-byparts}, in order to prove Theorem \ref{dilates}, it is enough to verify the inequality (\ref{approxlocal}) just for the function $f(x)=\langle x,n_x\rangle.$ Therefore, the application of Lemma \ref{cl2} finishes the proof. $\square$


\section{Appendix}

\begin{lemma}
Let $K$ be an origin-symmetric convex body and $w$ a continuous function on $S^{n-1}$. Then, \begin{align*}
    \lim_{\varepsilon\to 0}\frac{\mu(W(h_K + \varepsilon w)) - \mu(K)}{\varepsilon} &= \int_{S^{n-1}} w(\theta) d\sigma_{\mu, K}(\theta).
\end{align*}
\end{lemma}

\begin{proof} Our proof follows the proof given in the appendix of \cite{LivMink}. Recall that for $H_{n-1}-$almost every $x \in \partial K$ there exists a unique normal vector $n_x$. Let us denote the subset of $\partial K$ where this occurs by $\widetilde{\partial K}$. Let $X:\widetilde{\partial K} \times [0,\infty) \to \mathbb{R}^n\setminus K$ be defined by $X(x,t) = x + tn_x$, and let $D(x,t)$ be the Jacobian of this map. Moreover, from properties of Wulff shapes, we have that $h_{A[h_K + \varepsilon w]}(n_x) \le h_K(n_x) + \varepsilon w(n_x)$ with equality for $H_{n-1}-$almost every $x \in \widetilde{\partial K}$. See Section 7.5 in Schneider \cite{book4}. Let $\widetilde{\partial K}' \subset \widetilde{\partial K}$ be the subset where we have equality. Then, \begin{align*}
    \frac{1}{\varepsilon}(\mu(A[h_K + \varepsilon w]) - \mu(K)) &= \frac{1}{\varepsilon}\int_{\widetilde{\partial K}'}\int_{0}^{\varepsilon w(n_x)}D(x,t)g(x+tn_x)dt dH_{n-1}(x).
\end{align*} Observe that $X(x,t)$ is an expanding map. Indeed, for $x_1, x_2 \in \widetilde{\partial K}$ and $t_1, t_2 \in [0,\infty)$ we have \begin{align}\label{expanding}\begin{split}
    |X(x_1, t_1) - X(x_2, t_2)|^2 &= |x_1 + t_1 n_{x_1} - x_2 - t_2 n_{x_2}|^2 \\ &= |x_1 - x_2|^2 + |t_1n_{x_1} - t_2 n_{x_2}|^2 + t_1 \langle x_1 - x_2, n_{x_1}\rangle + t_2 \langle x_2-x_1, n_{x_2} \rangle.
\end{split}
\end{align} Since $K$ is convex, we have $\langle x_1, n_{x_1}\rangle \ge \langle x_2, n_{x_1}\rangle$ and $\langle x_2, n_{x_2}\rangle \ge \langle x_1, n_{x_2}\rangle$. Therefore, \begin{align*}
    |X(x_1, t_1) - X(x_2, t_2)| &\ge |x_1-x_2|^2 + |t_1n_{x_1} - t_2 n_{x_2}|^2 \\ &\ge |x_1-x_2|^2 + |t_1-t_2|^2
\end{align*} as desired. It follows that $D(x,t) \ge 1$, and so \begin{align}\label{lowerbound1}
\begin{split}
   \liminf_{\varepsilon \to 0}\frac{1}{\varepsilon}(\mu(A[h_K + \varepsilon w]) - \mu(K)) &\ge \liminf_{\varepsilon \to 0}\frac{1}{\varepsilon}\int_{\widetilde{\partial K}'}\int_{0}^{\varepsilon w(n_x)}g(x+tn_x) dt dH_{n-1}(x) \\ &= \int_{\widetilde{\partial K}'}w(n_x)g(x) dH_{n-1}(x).
\end{split}
\end{align} Since $\partial K\setminus \widetilde{\partial K}'$ has $H_{n-1}-$measure zero, we get that \begin{align}\label{lowerbound2}
\begin{split}
    \liminf_{\varepsilon \to 0}\frac{1}{\varepsilon}(\mu(A[h_K + \varepsilon w]) - \mu(K)) &\ge \int_{\partial K} w(n_x) g(x) dH_{n-1}(x) \\ &= \int_{S^{n-1}}w(\theta) d\sigma_{\mu, K}(\theta).
\end{split}
\end{align} We now pursue the reverse inequality. For an arbitrary $\delta > 0$, define \begin{align*}
    (\partial K)_{\delta} &= \{x \in \partial K: \exists \ a\in \mathbb{R}^n \text{ s.t. } x \in B(a,\delta) \subset K\}
\end{align*} where $B(a,\delta)$ is the Euclidean ball $\{y \in \mathbb{R}^n: |y-a|<\delta\}$. For a sufficiently small $\varepsilon>0$, take $0 \le t_1, t_2 \le \varepsilon$ and $x_1, x_2 \in (\partial K)_{\delta}$. From (\ref{expanding}), we have \begin{align*}
    |X(x_1, t_1) - X(x_2, t_2)| &\le |x_1-x_2|^2 + |t_1 - t_2|^2 + \varepsilon^2 |n_{x_1} - n_{x_2}|^2 + \varepsilon \langle x_1-x_2, n_{x_1} - n_{x_2}\rangle.
\end{align*} Now, it is a result of Hug \cite{Hug} that the Gauss map is Lipschitz on $(\partial K)_{\delta}$. Let us denote the Lipschitz constant by $L(\delta)$. Then \begin{align*}
    \frac{|x_1-x_2|^2+|t_1-t_2|^2+\varepsilon^2|n_{x_1} - n_{x_2}|^2 + \varepsilon \langle x_1 - x_2, n_{x_1} - n_{x_2}\rangle}{|x_1-x_2|^2 + |t_1-t_2|^2} &\le 1 + L(\delta)\varepsilon + L(\delta)^2 \varepsilon^2.
\end{align*} Hence, \begin{align*}
    D(x,t) &\le (1 + L(\delta)\varepsilon + L(\delta)^2 \varepsilon^2)^{n-1} \le 1 + C(K,n,\delta) \varepsilon.
\end{align*} We have therefore \begin{align*}
    \limsup_{\varepsilon\to 0}\frac{1}{\varepsilon}\int_{(\partial K)_{\delta} \cap \widetilde{\partial K}'}\int_{0}^{\varepsilon w(n_x)}D(x,t)g(x+tn_x)dt dH_{n-1}(x) &\le \int_{(\partial K)_{\delta} \cap \widetilde{\partial K}'}w(n_x) g(x) dH_{n-1}(x) \\ &= \int_{(\partial K)_{\delta}}w(n_x)g(x) dH_{n-1}(x).
\end{align*} Since $D(x,t) \ge 1$, we have as in (\ref{lowerbound1}) also that \begin{align*}
    \liminf_{\varepsilon \to 0}\frac{1}{\varepsilon}\int_{(\partial K)_{\delta} \cap \widetilde{\partial K}'}\int_{0}^{\varepsilon w(n_x)}D(x,t)g(x+tn_x)dt dH_{n-1}(x) &\ge \int_{(\partial K)_{\delta}}w(n_x) g(x) dH_{n-1}(x).
\end{align*} It follows that the limit in $\varepsilon$ exists and \begin{align*}
    \lim_{\varepsilon \to 0}\frac{1}{\varepsilon}\int_{(\partial K)_{\delta} \cap \widetilde{\partial K}'}\int_{0}^{\varepsilon w(n_x)}D(x,t)g(x+tn_x)dt dH_{n-1}(x) &= \int_{(\partial K)_{\delta}}w(n_x) g(x) dH_{n-1}(x).
\end{align*} By the dominated convergence theorem and lower semi-continuity, \begin{align*}
    \limsup_{\varepsilon\to 0}\frac{1}{\varepsilon}(\mu(A[h_K + \varepsilon w]) &- \mu(K)) = \limsup_{\varepsilon\to 0}\frac{1}{\varepsilon}\int_{\widetilde{\partial K}'}\int_{0}^{\varepsilon w(n_x)}D(x,t)g(x+tn_x)dt dH_{n-1}(x) \\ &= \limsup_{\varepsilon\to 0}\lim_{\delta \to 0}\frac{1}{\varepsilon}\int_{(\partial K)_{\delta} \cap \widetilde{\partial K}'}\int_{0}^{\varepsilon w(n_x)}D(x,t)g(x+tn_x)dt dH_{n-1}(x) \\ &= \lim_{\delta \to 0}\lim_{\varepsilon\to 0}\frac{1}{\varepsilon}\int_{(\partial K)_{\delta} \cap \widetilde{\partial K}'}\int_{0}^{\varepsilon w(n_x)}D(x,t)g(x+tn_x)dt dH_{n-1}(x) \\ &= \lim_{\delta \to 0}\int_{(\partial K)_{\delta}}w(n_x)g(x) dH_{n-1}(x) \\ &= \int_{\widetilde{\partial K}}w(n_x)g(x)dH_{n-1}(x) \\ &= \int_{\partial K}w(n_x)g(x)dH_{n-1}(x) \\ &= \int_{S^{n-1}}w(\theta) d\sigma_{\mu, K}(\theta).
\end{align*} Combining this with (\ref{lowerbound2}) gives us the desired conclusion.
\end{proof}

\end{document}